\documentclass{article}

\usepackage{hyperref}
\usepackage{nomencl}
\usepackage{chngcntr}
\usepackage{amsmath}
\usepackage{amsthm}
\usepackage{amscd}
\usepackage{bbm}
\usepackage{enumerate}
\usepackage{amssymb}
\usepackage{palatino}
\usepackage{amscd}
\usepackage{verbatim}
\usepackage{mathrsfs}
\usepackage{esint}
\usepackage{xcolor}

\usepackage{alphabeta}

\usepackage[T2A,T1]{fontenc}
\usepackage[utf8]{inputenc}

\DeclareSymbolFont{cyrillic}{T2A}{cmr}{m}{n}
\def\makecyrsymbol#1#2{%
  \begingroup\edef\temp{\endgroup
    \noexpand\DeclareMathSymbol{\noexpand#1}
    {\noexpand\mathalpha}{cyrillic}%
    {\expandafter\expandafter\expandafter
     \calccyr\expandafter\meaning\csname T2A\string#2\endcsname\end}}%
  \temp}
\expandafter\def\expandafter\calccyr\string\char#1\end{#1}

\makecyrsymbol\hardsign\cyrhrdsn
\makecyrsymbol\bighardsign\CYRHRDSN
\makecyrsymbol\iy\cyrery
\makecyrsymbol\bigiy\CYRERY
\makecyrsymbol\softsign\cyrsftsn
\makecyrsymbol\bigsoftsign\CYRSFTSN

\newtheoremstyle{dotless}{}{}{\itshape}{}{\bfseries}{}{ }{}

\newtheorem{thm}{Theorem}
\newtheorem{lem}[thm]{Lemma}

\newtheorem{conj}[thm]{Conjecture}

\theoremstyle{dotless}

\counterwithin{equation}{section}
\counterwithin{thm}{section}

\renewcommand{\Re}{\operatorname{\mathfrak{Re}}}
\renewcommand{\Im}{\operatorname{\mathfrak{Im}}}

\newcommand{\Sym}{\mathrm{Sym}}

\newcommand{\Z}{\mathbb{Z}}
\newcommand{\R}{\mathbb{R}}
\newcommand{\Q}{\mathbb{Q}}

\newcommand{\F}{\mathbb{F}}

\newcommand{\id}{\mathop{\mathrm{id}}}

\newcommand{\Aut}{\mathop{\mathrm{Aut}}}

\newcommand{\Gal}{\mathop{\mathrm{Gal}}}

\newcommand{\inj}{\hookrightarrow}

\newcommand{\Avg}{\mathop{\mathrm{Avg}}}

\newcommand{\SL}{\mathrm{SL}}
\newcommand{\GL}{\mathrm{GL}}

\newcommand{\actson}{\curvearrowright}

\newcommand{\mat}[4]{\left(\begin{array}{cc} #1 & #2\\ #3 & #4\end{array}\right)}
\newcommand{\twobytwo}[4]{\mat{#1}{#2}{#3}{#4}}

\newcommand{\eps}{\varepsilon}
\newcommand{\disc}{\mathrm{disc}}

\newcommand{\A}{\mathbb{A}}

\newcommand{\K}{\mathbb{K}}

\newcommand{\Qbar}{\overline{\Q}}
\renewcommand{\d}{\partial}

\renewcommand{\P}{\mathbb{P}}

\newcommand{\diag}{\mathrm{diag}}

\newcommand{\PGL}{\mathrm{PGL}}

\newcommand{\Sel}{\mathrm{Sel}}

\newcommand{\Fbar}{\overline{\F}}

\newcommand{\smalltwobytwo}[4]{\left(\begin{smallmatrix} #1 & #2\\ #3 & #4\end{smallmatrix}\right)}
\newcommand{\SO}{\mathrm{SO}}
\newcommand{\nontriv}{\mathrm{nontriv.}}
\newcommand{\supp}{\mathop{\mathrm{supp}}}

\newcommand{\Inv}{\mathrm{Inv}}

\let\uglyphi\phi
\let\phi\varphi

\let\emptyset\varnothing

 \DeclareFontFamily{U}{wncy}{}
    \DeclareFontShape{U}{wncy}{m}{n}{<->wncyr10}{}
    \DeclareSymbolFont{mcy}{U}{wncy}{m}{n}
    \DeclareMathSymbol{\Sha}{\mathord}{mcy}{"58} 

\usepackage{babel}
\usepackage{epigraph}

\newcommand{\greekfont}[1]{#1}
\renewcommand{\greekfont}[1]{{\fontfamily{txr}\selectfont #1}}

\title{Quadrics in arithmetic statistics.}

\author{\Large Levent Alp\"{o}ge\\~\\\text{\greekfont{\large \foreignlanguage{greek}{Εις μνήμην της Χρυσής Νότσκα.}}}}

\date{}

\makeatletter
\DeclareRobustCommand{\pmods}[1]{\mkern4mu({\operator@font mod}\mkern6mu#1)}
\makeatother

\begin{document}

\maketitle

\renewcommand{\abstractname}{Abstract.}

\begin{abstract}
We (re)introduce the circle method into arithmetic statistics.

More specifically, we combine the circle method with Bhargava's counting technique in order to give a general method that allows one to treat arithmetic statistical problems in which one is trying to count orbits on a subvariety of affine space defined by the vanishing of a quadratic invariant.

We explain this method by way of example by sketching how to compute the average size of $2$-Selmer groups in each of the families $y^2 = x^3 + B^k$, giving full details in the case $k\equiv 1\pmods{6}$.

In the course of the argument we introduce a smoothed form of Bhargava's aforementioned method, as well as a trick which allows us to compute the above averages from knowledge of the averages over "unconstrained" families.
\end{abstract}

\section{Introduction.}
The field of arithmetic statistics lacks a general technique to count orbits on nontrivial invariant subvarieties. Essentially every\footnote{We know of exactly one work that faces point counting on a nontrivial subvariety "directly", namely Samuel Ruth's Ph.D.\ thesis \cite{chapterthree-ruth's-thesis}. However, there are certainly others, e.g.\ Yao Xiao's \cite{yao-xiao}, that change such counting problems into more tractable ones by using particular features of their situations. While quite clever, we believe they are not relevant here because our goal is to produce a general technique (awkward as it is to formulate general statements in this context).} application of the counting technique invented by Bhargava in his Ph.D.\ thesis has reduced such an orbit counting problem to one of counting lattice points in an expanding region and then dealt\footnote{We are of course abbreviating significantly.} with the latter by invoking a standard lemma of Davenport, which we will think of as a jazzed up version of the usual count of integral points of bounded height in affine space.

The purpose of this paper is to give the first general technique going beyond this regime.

The technique arises from a substantial simplification we give of the argument in Samuel Ruth's Ph.D.\ thesis \cite{chapterthree-ruth's-thesis}, in which he bounds the average size of $2$-Selmer groups in the family of Mordell curves $y^2 = x^3 + B$. The reason his argument needs simplification is that his treatment of what we will call the "tail" is quite hard-going and in any case particular to his problem.

That said, his idea of combining the circle method with Bhargava's counting method makes what we will call the "bulk" contribution trivial to deal with, and this is our starting point. We then replace his treatment of the "tail" with a trick. Finally we forego any calculation of local densities with another trick --- by formally deducing that the $2$-Selmer average in question must match the $2$-Selmer average over the "unconstrained" family of elliptic curves $y^2 = x^3 + Ax + B$.

A slightly less informal summary of the technique is as follows. We follow Ruth in using the circle method to count said points in the bulk (namely when not "polynomially high in the cusp") by using the smoothed delta symbol method, and then we provide an a priori upper bound for the point count in the tail (namely when polynomially high in the cusp) via a divisor bound. As usual the divisor bound produces an estimate which is subpolynomially worse than sharp --- however the key point is that this loss is more than compensated for by the "overconvergence" of the volume integral (and the fact that we only apply said estimate when polynomially high in the cusp). As for how we avoid any calculation of constants, note that the circle method already outputs a product of local densities, which are informally calculated by "thickening" an equality to a congruence modulo a highly divisible integer $N$ (the singular integral has the analogous property) and taking $N\to \infty$. Thus our $2$-Selmer average, aka the $2$-Selmer average over the family $y^2 = x^3 + Ax + B$ with the equality $A = 0$ imposed, is $3$ because the $2$-Selmer average over the family $y^2 = x^3 + Ax + B$, with only the congruence $A\equiv 0\pmods{N}$ imposed, is $3$ (independent of $N$!), by work of Bhargava-Shankar.

Now let us give a precise statement of the example application.

\subsection{Main theorem.}

\begin{thm}\label{all six families}
Let $k\in \Z$. Let $\mathcal{B}\subseteq \Z - \{0\}$ be a set of positive density defined by congruence conditions. Then: $$\Avg_{B\in \mathcal{B} : |B|\leq X} \#|\Sel_2(E_{0,B^k}/\Q)|\leq \left(\begin{cases} 3 & k\not\equiv 0\pmods{3}\\ 2 & k\equiv 0\pmods{6}\\ \infty & k\equiv 3\pmods{6}\end{cases}\right) + O_{\mathcal{B}}(o_{X\to \infty}(1)),$$
with equality if e.g.\footnote{We write "with equality if e.g." to emphasize that this is not an if and only if statement.}\ $\mathcal{B}\subseteq \Z-\{0\}$ is defined by finitely many congruence conditions.
\end{thm}\noindent
Here a subset $\mathcal{B}\subseteq \Z - \{0\}$ is defined by congruence conditions if and only if $\mathcal{B}\cap \Z^+\subseteq \Z^+$ and $(-\mathcal{B})\cap \Z^+\subseteq \Z^+$ are, and a subset $\mathcal{B}\subseteq \Z^+$ is defined by congruence conditions if and only if for all $p$ there is an open subset $\mathcal{B}_p\subseteq \Z_p$ such that $\mathcal{B} = \Z^+\cap \prod_p \mathcal{B}_p$ as subsets of $\prod_p \Z_p$. It is straightforward to adapt our arguments to prove equality in Theorem \ref{all six families} for the more general class of subsets $\emptyset\neq \mathcal{B}\subseteq \Z - \{0\}$ which are "large" in a sense completely analogous to that of Bhargava-Shankar's \cite{chaptertwo-bhargava-shankar-2-selmer}, but we will spare ourselves the notational effort.

There is of course nothing to do when $k\equiv 0\pmods{3}$: the statement is evident if $k\equiv 0\pmods{6}$ and follows from the fact that $\#|\Sel_2(E_{0,B^k}/\Q)|\geq 2^{\omega_{7\bmod{12}}(B) - \omega_{11\bmod{12}}(B)}$ (consider Tamagawa factors for the evident $2$-isogeny) when $k\equiv 3\pmods{6}$, where $\omega_{a\bmod{b}}(n) := \#|\{p\mid n : p\equiv a\pmods{b}\}|$.

In the case $k\equiv 1\pmods{6}$ Theorem \ref{all six families} is a slight generalization of the main theorem of Ruth's Ph.D.\ thesis \cite{chapterthree-ruth's-thesis} (see Theorem $1.1.2$ of his \cite{chapterthree-ruth's-thesis}, which amounts to a proof of the upper bound when $\mathcal{B} = \Z - \{0\}$) that was important for \cite{cubic-paper-with-ari-and-manjul} (and thus \cite{quartic-paper-with-ari-and-manjul}).

In our eyes it is also this case that serves as the clearest example of the method. Indeed, tracking the condition "$A = 0$" through the arguments of Bhargava-Shankar \cite{chaptertwo-bhargava-shankar-2-selmer}, we find that we must count orbits of binary quartic forms $F(X,Y)\in \Z[X,Y]$ with classical invariants $I(F) = 0$ and $|J(F)|$ bounded. We are counting points on a quadric because, writing $F(X,Y) =: a\cdot X^4 + b\cdot X^3 Y + c\cdot X^2 Y^2 + d\cdot X Y^3 + e\cdot Y^4$, $$I(F) = 12ae - 3bd + c^2.$$

We will sketch the remaining cases. In the case $k\equiv 2\pmods{6}$, as in Bhargava-Ho \cite{chapterthree-bhargava-ho-website} the relevant parametrization of $2$-Selmer elements is by orbits of pairs of binary cubic forms $(F_1, F_2)$ with $F_i\in \Z[X,Y]$, and now the invariant quadric is $0 = I(F_1,F_2) = 3 a_1 d_2 - b_1 c_2 + c_1 b_2 - 3 d_1 a_2$, where $F_i(X,Y) =: a_i\cdot X^3 + \cdots + d_i\cdot Y^3$, but otherwise is --- at a high level --- similar to the case $k\equiv 1\pmods{6}$, but is more intricate and dealt with in joint work with Bhargava and Shnidman \cite{sum-of-two-cubes-paper-with-ari-and-manjul}. What is perhaps different is our deduction of the cases $k\equiv 4,5\pmods{6}$ by rescaling our orbits suitably (the difficulty being that one must preserve integrality) and then appealing to results already proven in the corresponding case $k\equiv 3\pmods{6}$ --- an argument that ultimately works because $E_{0,B^k}$ and $E_{0,B^{k+3}}$ are quadratic twists, and so their $2$-Selmer groups are both subsets of e.g.\ $H^1(\Gal(\Qbar/\Q), E_{0,B^k}[2])$ defined by --- potentially different, of course --- local conditions. (Note that this trick already occurs in \cite{sum-of-two-cubes-paper-with-ari-and-manjul}, and, though our treatment is explicit and does not manifestly involve this cohomological interpretation, we certainly learned said interpretation from that joint work.) In the case $k\equiv 4\pmods{6}$ we will also explain how to arrive to the same result through another explicit parametrization of $\Sel_2(E_{0,B^4}/\Q)$ by binary quartic forms. Note that the case $k\equiv 4\pmods{6}$ is treated in yet another way --- in fact by reducing to results already proven in the case $k\equiv 2\pmods{6}$ instead --- in the aforementioned joint work \cite{sum-of-two-cubes-paper-with-ari-and-manjul}.

Note also that via the aforementioned trick one deduces the statement of Theorem \ref{all six families} for the slightly more general families $\{y^2 = x^3 + d B^k : B\in \mathcal{B}\}$ with $k\in \Z$, $0\neq d\in \Z$ and $\mathcal{B}\subseteq \Z - \{0\}$ (quadratic twist by $d$ in the nontrivial case $k\not\equiv 0\pmods{3}$).

\subsection{Main technique.}

Having stated the main theorem, let us discuss the method of proof. In order to be specific, let us put ourselves in the case $k\equiv 1\pmods{6}$: we would like to count, up to $\PGL_2(\Q)$-equivalence, the locally soluble binary quartic forms $F\in \Z[X,Y]$ with $I(F) = 0$ and $0\neq |J(F)|\leq X$. Write then $V := \Sym^4(2)$ for the space of binary quartic forms. It is now standard that, using a method introduced in Bhargava's Ph.D.\ thesis \cite{chapterthree-bhargava-thesis} and first carried out in this context by Bhargava-Shankar \cite{chaptertwo-bhargava-shankar-2-selmer} (let us ignore our smoothing for the sake of this discussion), the problem immediately reduces to the problem of obtaining an asymptotic of the following form:
\begin{align*}
&\int_{1\ll \lambda\ll X^{\frac{1}{24}}} d^\times \lambda\, \int_{|u|\ll 1} du\, \int_{1\ll t\ll \lambda} t^{-2} d^\times t\, \#|\{F\in \lambda\cdot n_u\cdot a_t\cdot G_0\cdot L\cap V(\Z)^{\mathrm{irred.}} : I(F) = 0\}|
\\&\sim \mathrm{const.}\cdot X^{\frac{1}{2}}.
\end{align*}\noindent
Here $n_u := \smalltwobytwo{1}{0}{u}{1}$, $a_t := \smalltwobytwo{t^{-1}}{0}{0}{t}$, and $d^\times z := \frac{dz}{z}$, so that $d^\times \lambda\, du\, t^{-2} d^\times t\, dk$ is the Haar measure on $\GL_2^+(\R)$ (the "$+$" denoting positive determinant) induced by Haar measures on $\Lambda, N, A, K$ under the Iwasawa decomposition $\GL_2^+(\R) = \Lambda\cdot N\cdot A\cdot K$, where $\Lambda$ are the scalars, $N$ are the lower unipotents, $A$ is the diagonal torus, and $K := \SO_2(\R)$ is the maximal compact.

For the sake of exposition, the reader should imagine $G_0\cdot L\subseteq V(\R)$ as a small ball of binary quartic forms of height $\asymp 1$, and indeed should imagine (though also keep in mind that this is a slight oversimplification) that the set $\lambda\cdot n_u\cdot a_t\cdot G_0\cdot L\cap V(\Z)^{\mathrm{irred.}}$ is simply the set $S_{\mathrm{approx.}}$, say, of $(a,b,c,d,e)\in \Z$ satisfying $a,e\neq 0$ and:
\begin{align*}
|a|&\ll \frac{\lambda^4}{t^4},
\\|b|&\ll \frac{\lambda^4}{t^2},
\\|c|&\ll \lambda^4,
\\|d|&\ll t^2\cdot \lambda^4,
\\|e|&\ll t^4\cdot \lambda^4.
\end{align*}

Examining the integral, one evidently reduces to needing to treat $$\#|\{F\in \lambda\cdot n_u\cdot a_t\cdot G_0\cdot L\cap V(\Z)^{\mathrm{irred.}} : I(F) = 0\}|.$$ Were there no condition $I(F) = 0$, this would essentially be counting lattice points in a somewhat "round" set scaled by a parameter $\lambda$, which is simple --- in the standard treatment one invokes an easy lemma of Davenport, which is precisely tailored for this sort of counting problem. Indeed, examining the defining inequalities of $S_{\mathrm{approx.}}$, it is evidently quite simple to evaluate $\#|S_{\mathrm{approx.}}|$.

Thus the entire difficulty is in dealing with the condition $I(F) = 0$. However we are simply asking to count zeroes of a quadratic form in five variables that lie in the scalings of a somewhat "round" set, and this is easy if the set is in fact quite "round" --- in other words, if we are in the "bulk" and $1\ll t\ll_\eps \lambda^\eps$, say, then the defining inequalities of $S_{\mathrm{approx.}}$ are all, up to $X^{O(\eps)}$, the same, so that $S_{\mathrm{approx.}}$ is essentially a scaled cube, and it is straightforward to count the number of solutions of a quadratic form in five variables lying in such a set, by the oldest forms of the circle method. For convenience we follow Ruth and use the smoothed delta symbol method, which trivializes the problem. In the end we obtain the asymptotic $$\#|\{F\in \lambda\cdot n_u\cdot a_t\cdot B\cdot L\cap V(\Z)^{\mathrm{irred.}} : I(F) = 0\}| = \mathrm{const.}\cdot \lambda^{12} + O_\eps(\lambda^{12 - \delta + O(\eps)})$$ for a positive absolute constant $\delta > 0$.

Thus the entire problem has reduced to treating the count when we are in the "tail", i.e.\ $\lambda^\eps\ll_\eps t\ll \lambda$. Here Ruth uses an argument that again involves the smoothed delta symbol method and is quite specific to the situation.

Our observation is that the "tail" case is also obviously trivial. Specifically, it is obvious that $$\#|\{F\in \lambda\cdot n_u\cdot a_t\cdot B\cdot L\cap V(\Z)^{\mathrm{irred.}} : I(F) = 0\}|\ll \lambda^{12 + o(1)}$$ by the divisor bound: $12ae - 3bd + c^2 = 0$ implies that $a,e\,\vert\, 3bd - c^2$. Because we are dealing with elements in $V(\Z)^{\mathrm{irred.}}$ (and thus $a,e\neq 0$), it follows that $(b,c,d)$ determine $(a,e)$ up to $\ll \lambda^{o(1)}$ choices, and the number of $(b,c,d)$ is\footnote{While we have seemingly used the special form of $I(F)$ in this argument, this divisor bound argument works in general --- the point is that in a box a binary quadratic form represents a nonzero integer few times, because of a divisor bound in at most a quadratic extension (and the usual proof of Dirichlet's unit theorem to deal with units).} $\ll \lambda^{12}$.

Because of the $o(1)$ in the exponent, this bound is not quite as sharp as what we obtained in the "bulk", but since we are in the "tail" we needn't work so hard. The point is that the condition $t\gg_\eps \lambda^\eps$ combines with the $t^{-2} d^\times t$ in the Haar measure (note that the exponent $2$ is more than is needed for convergence, so to speak) to imply that this bound is enough to give a bound $\ll_\eps \lambda^{12 - \Omega(\eps)}$ on the "tail" contribution after integrating over $u$ and $t$.

So in the end we obtain the desired asymptotic, which is equivalent to the theorem via Bhargava's counting method. The argument in the case of pairs of binary cubic forms is the same, and indeed the above argument easily generalizes to other quadrics in at least four variables (said restriction arising in order to ensure that the circle method analysis in the "bulk" goes through easily).

\section{Acknowledgments.}
Chrysi Notskas' influence on my life, mathematical and otherwise, was profound and this work is dedicated to her memory.\footnote{It was through her that I was introduced to higher mathematics (and also to parallel programming and computer vision)! I found her phrase "bigger and better things", regarding how to respond to the unjustifiable, unforgettable.} This article is based on Chapter $3$ of the author's Ph.D.\ thesis at Princeton University \cite{my-thesis}. I would like to thank both my advisor Manjul Bhargava and Peter Sarnak for their patience and encouragement. I would also like to thank Joseph Glynias, Ari Shnidman, Odysseus Skartsis, and Jacob Tsimerman for informative discussions. Finally I thank the National Science Foundation (via their grant DMS-$2002109$), Columbia University, and the Society of Fellows for their support during the pandemic.

It has been indicated to me that it is worth clarifying the following. As usual I will not be giving this paper to some journal. In contrast, a complete and detailed proof of Theorem \ref{all six families} for $k\equiv 2\pmods{6}$ (cf.\ the sketch given in Section \ref{square mordell curves section}) will appear (and thus be journal published along with the analytic method itself), in the joint work \cite{sum-of-two-cubes-paper-with-ari-and-manjul}.

\section{Smoothing in Bhargava's counting method.}\label{smoothing section}

Let us now introduce the aforementioned technical convenience that simplifies Bhargava's counting method, as first introduced in Chapter $5$ of his Ph.D.\ thesis \cite{chapterthree-bhargava-thesis}, though the trick of averaging over fundamental domains was first introduced in the published version (see Section $2.2$ of Bhargava's \cite{chapterthree-bhargava-quartic-discriminants}). The point is that, while an average over $G_0$ (using the notation of the above outline) improves the situation considerably, one is still integrating a "rough" function, namely $\mathbbm{1}_{G_0}$, and it is wiser to instead integrate a compactly supported smooth function. We note that this is completely natural from Bhargava's original formulation --- see equation $(4)$ in Section $2.2$ of Bhargava's \cite{chapterthree-bhargava-quartic-discriminants}, and note that we are taking (in his notation) $\Phi$ to be smooth and compactly supported, rather than the indicator function of a box.

In order to be specific, and for the reader's convenience, let us work in the setup of the proof of Theorem \ref{all six families} for $k\equiv 1\pmods{6}$ (it will be clear how to modify the construction for Theorem $k\equiv 2\pmods{6}$, and indeed in any application of Bhargava's counting method). That is, $V := \Sym^4(2), G := \GL_2$, $L := L^{(0)}\coprod L^{(1)}\coprod L^{(2+)}\coprod L^{(2-)}$ with
\begin{align*}
L^{(0)} &:= \left\{X^3 Y - \frac{1}{3}\cdot X Y^3 + \frac{J}{27}\cdot Y^4 : J\in (-2,2)\right\},\\
L^{(1)} &:= \left\{X^3 Y - \frac{I}{3}\cdot X Y^3 \pm \frac{2}{27}\cdot Y^4 : I\in [-1,1)\right\}\cup \left\{X^3 Y + \frac{1}{3}\cdot X Y^3 + \frac{J}{27}\cdot Y^4 : J\in (-2,2)\right\},\\
L^{(2\pm)} &:= \pm \left\{\frac{1}{16} X^4 - \sqrt{\frac{2-J}{27}}\cdot X^3 Y + \frac{1}{2}\cdot X^2 Y^2 + Y^4 : J\in (-2,2)\right\},
\end{align*}\noindent
and $$\mathcal{F} := \left\{\lambda\cdot n_u\cdot a_t\cdot k : \lambda\in \R^+, u\in \nu(t)\subseteq \left[-\frac{1}{2}, \frac{1}{2}\right], t\geq \sqrt{\frac{\sqrt{3}}{2}}, k\in \SO_2(\R)\right\}\subseteq G(\R),$$ Gauss's classical fundamental domain for $G(\Z)\actson G(\R)$. Note that $L$ is a fundamental domain for $G(\R)\actson V(\R)^{\Delta\neq 0}$. Observe that the points with $I = 0$ and $\Delta\neq 0$ are all in the $G(\R)$-orbit of the two forms $F_{\pm}(X,Y) := X^3 Y \pm \frac{2}{27}\cdot Y^4$, which lie in the interior of $L^{(1)}$ (and thus lie in small compact subintervals thereof). Thus the following setup suffices for us.

Write, for each $v\in V(\R)^{\Delta\neq 0}$, $v_L\in L$ for the unique element of $L$ mapping to the image of $v$ under $V(\R)^{\Delta\neq 0}\to V(\R)^{\Delta\neq 0}/G(\R)\simeq L$.

Let $\mu_\pm\in C_c^\infty(\R)$ be such that $\mu_-\leq \mathbbm{1}_{[-1,1]}\leq \mu_+$. Let $\alpha\in C_c^\infty(G(\R))$ and $\beta\in C_c^\infty(L)$ be compactly supported smooth functions such that: $\alpha$ is $\SO_2(\R)$-invariant, $\int_{G(\R)} \alpha = 1$, $\beta(F_{\pm}) = 1$, and $\supp{\beta}, \beta^{-1}(\{1\})\subseteq L^{(1)}$ are both unions of two small compact intervals respectively containing $F_\pm$. Let $$\phi(v) := \sum_{g\cdot v_L = v} \alpha(g)\cdot \beta(v_L).$$ Note that this is a finite sum because stabilizers of elements of $V(\R)^{\Delta\neq 0}$ are finite.

Via $\mu_\pm$, $\alpha$, and $\beta$ we get a slightly more convenient way\footnote{Specifically, this insertion of a smooth weight in Bhargava's main trick in his counting technique saves us the effort required to remove smooth weights when applying the smoothed delta symbol method. We note here that we use $\beta$ to ensure smoothness of $\phi$ --- note that $L^{(0)}\coprod L^{(1)}$ is a rectangle missing two corners, and at the other two corners a similar definition of $\phi$ with $\beta = 1$ identically would fail to be smooth. One can get around this, of course, but this choice simplifies notation.} to smooth out the various integrals in Bhargava's counting technique --- instead of integrating the normalized indicator function $\frac{1}{\int_{G_0} dg}\cdot \mathbbm{1}_{G_0}$ of $G_0$ over $g\in G(\R)$ (i.e.\ integrating over $g\in G_0$) and observing that $g\cdot \mathcal{F}$ is also (the closure of) a fundamental domain for $G(\Z)\actson G(\R)$ so that all counts are independent of $g$, we instead integrate $\alpha(g)$ over $g\in G$ and then make the same observation.

Specifically, we observe that, since $$\#|\{F\in \mathcal{F}\cdot g\cdot L\cap V(\Z)^\nontriv : I(F) = 0, 0\neq |J(F)|\leq X\}|$$ is constant in $g$ outside a measure-zero subset of $G(\R)$ (since $\mathcal{F}\cdot g$ and $\mathcal{F}$ are both fundamental domains for $G(\Z)\actson G(\R)$ --- the first main observation of Section $2.2$ of Bhargava's \cite{chapterthree-bhargava-quartic-discriminants}), it follows that:
\begin{align*}
&\frac{\int_{g\in G_0} dg\, \#|\{F\in \mathcal{F}\cdot g\cdot L\cap V(\Z)^\nontriv : I(F) = 0, 0\neq |J(F)|\leq X\}|}{\int_{g\in G_0} dg}
\\&= \int_{G(\R)} dg\, \alpha(g)\cdot \#|\{F\in \mathcal{F}\cdot g\cdot L\cap V(\Z)^\nontriv : I(F) = 0, 0\neq |J(F)|\leq X\}|.
\end{align*}\noindent
The left-hand side is precisely Ruth's $N(Y^{\mathrm{irr}}, X)$.

We then manipulate this expression just as in Section $2.3$ of Bhargava-Shankar's \cite{chaptertwo-bhargava-shankar-2-selmer} (and implicitly in Bhargava's \cite{chapterthree-bhargava-quartic-discriminants}).

Evidently:
\begin{align*}
&\int_{G(\R)} dg\, \alpha(g)\cdot \#|\{F\in \mathcal{F}\cdot g\cdot L\cap V(\Z)^\nontriv : I(F) = 0, 0\neq |J(F)|\leq X\}|
\\&= \sum_{v\in V(\Z)^\nontriv : I(v) = 0, 0\neq |J(v)|\leq X} \int_{G(\R)} dh\, \alpha(h)\cdot \#|\{g\in \mathcal{F} : gh\cdot v_L = v\}|.
\end{align*}

But, using that $\beta(v_L) = 1$ if $I(v) = 0$ and $\Delta(v)\neq 0$ (since this implies that $v_L = F_+$ or $F_-$), for $v\in V(\Z)^\nontriv$ with $I(v) = 0$ and $J(v)\neq 0$ each integral can be evaluated as follows:
\begin{align*}
&\int_{G(\R)} dh\, \alpha(h)\cdot \#|\{g\in \mathcal{F} : gh\cdot v_L = v\}|
\\&= \sum_{\gamma\in G(\R): \gamma\cdot v_L = v} \int_{G(\R)} dh\, \alpha(h)\cdot \beta(v_L)\cdot \#|\{g\in \mathcal{F} : gh = \gamma\}|
\\&= \sum_{\gamma\in G(\R) : \gamma\cdot v_L = v} \int_{\mathcal{F}^{-1}\cdot \gamma} dh\, \alpha(h)\cdot \beta(v_L)
\\&= \sum_{\gamma\in G(\R) : \gamma\cdot v_L = v} \int_{\mathcal{F}^{-1}} dh\, \alpha(h\cdot \gamma)\cdot \beta(v_L)
\\&= \int_{\mathcal{F}} dh\, \sum_{\gamma\in G(\R) : \gamma\cdot v_L = v} \alpha(h^{-1}\cdot \gamma)\cdot \beta(v_L)
\\&= \int_{\mathcal{F}} dh\, \sum_{g\in G(\R) : g\cdot v_L = h^{-1}\cdot v} \alpha(g)\cdot \beta(v_L)
\\&= \int_{\mathcal{F}} dh\, \phi(h^{-1}\cdot v),
\end{align*}\noindent
by definition.

Therefore we have found that:
\begin{align*}
N(Y(\Z)^\nontriv, X) &:= \#|\{F\in \mathcal{F}\cdot L\cap V(\Z)^\nontriv : I(F) = 0, 0\neq |J(F)|\leq X\}|
\\&= \sum_{v\in V(\Z)^\nontriv : I(v) = 0, 0\neq |J(v)|\leq X} \int_{\mathcal{F}} dh\, \phi(h^{-1}\cdot v)
\\&= \int_{\mathcal{F}} dh\, \sum_{v\in V(\Z)^\nontriv : I(v) = 0, 0\neq |J(v)|\leq X} \phi(h^{-1}\cdot v),
\end{align*}\noindent
and so e.g. $N_{\mu_-}(Y(\Z)^\nontriv, X)\leq N(Y(\Z)^\nontriv, X)\leq N_{\mu_+}(Y(\Z)^\nontriv, X)$, where:
$$N_{\mu_\pm}(Y(\Z)^\nontriv, X) := \int_{\mathcal{F}} dh\, \sum_{v\in V(\Z)^\nontriv : I(v) = 0, J(v)\neq 0} \mu_\pm\left(\frac{J(v)}{X}\right)\cdot \phi(h^{-1}\cdot v).$$

\section{Proof of Theorem \ref{all six families} for $k\equiv 1\pmods{6}$.}\label{mordell curves section}

Now let us turn to the proof of Theorem \ref{all six families} when $k\equiv 1\pmods{6}$.

\subsection{Reduction to point counting.}

For the reader's convenience we use Ruth's notation in what follows. Let $V := \Sym^4(2)$. Let $G := \GL_2$. Let, for $F\in V$,
\begin{align*}
I(F) &:= 12ae - 3bd + c^2,
\\ J(F) &:= 72ace + 9bcd - 27ad^2 - 27b^2 e - 2c^3,
\end{align*}\noindent
so that the discriminant $$\Delta(F) = 4I^3 - J^2,$$ where we have written $F(X,Y) =: a\cdot X^4 + b\cdot X^3 Y + c\cdot X^2 Y^2 + d\cdot X Y^3 + e\cdot Y^4$. Let $G\actson V$ via $(g\cdot F)(X,Y) := F((X,Y)\cdot g)$. We note that, for $g\in G$ and $F\in V$, we have that:
\begin{align*}
I(g\cdot F) &= (\det{g})^4\cdot I(F),
\\ J(g\cdot F) &= (\det{g})^6\cdot J(F).
\end{align*}

Let $L := L^{(0)}\coprod L^{(1)}\coprod L^{(2+)}\coprod L^{(2-)}$ with
\begin{align*}
L^{(0)} &:= \left\{X^3 Y - \frac{1}{3}\cdot X Y^3 + \frac{J}{27}\cdot Y^4 : J\in (-2,2)\right\},\\
L^{(1)} &:= \left\{X^3 Y - \frac{I}{3}\cdot X Y^3 \pm \frac{2}{27}\cdot Y^4 : I\in [-1,1)\right\}\cup \left\{X^3 Y + \frac{1}{3}\cdot X Y^3 + \frac{J}{27}\cdot Y^4 : J\in (-2,2)\right\},\\
L^{(2\pm)} &:= \pm \left\{\frac{1}{16} X^4 - \sqrt{\frac{2-J}{27}}\cdot X^3 Y + \frac{1}{2}\cdot X^2 Y^2 + Y^4 : J\in (-2,2)\right\},
\end{align*}\noindent
a fundamental domain for $G(\R)\actson V(\R)^{\Delta\neq 0}$.

Let $$\mathcal{F} := \left\{\lambda\cdot n_u\cdot a_t\cdot k : \lambda\in \R^+, u\in \nu(t), t\geq \sqrt{\frac{\sqrt{3}}{2}}, k\in \SO_2(\R)\right\}\subseteq G(\R),$$ where $\nu(t)\subseteq [-\frac{1}{2},\frac{1}{2}]$ is $[-\frac{1}{2}, \frac{1}{2})$ when $t\gg 1$ or else a union of two subintervals of $[-\frac{1}{2}, \frac{1}{2}]$ when $\sqrt{\frac{\sqrt{3}}{2}}\leq t\ll 1$ (just imagine the usual Gauss fundamental domain for $\SL_2(\Z)\actson \mathfrak{h}$, and note that $t^2$ corresponds to $\Im{\tau}$ and $u$ corresponds to $\Re{\tau}$). Here we have written
\begin{align*}
n_u &:= \twobytwo{1}{0}{u}{1},
\\a_t &:= \twobytwo{t^{-1}}{0}{0}{t}.
\end{align*}

We note that $\mathcal{F}$ is a fundamental domain for $G(\Z)\actson G(\R)$.

Let $\delta\in \R^+$ be a small absolute constant (e.g.\ $\delta := 10^{-10^{10}}$ will certainly suffice). Let $\eta\in \R$ (we will eventually take $\eta := \pm X^{-\delta^2}$). Let $\mu: \R\to [0,1]$ be $C^\infty$, such that $\mu = 0$ outside $[-1-2\eta, 1+2\eta]$, and such that $\mu = 1$ on $[-1-2\eta+|\eta|, 1+2\eta-|\eta|]$ --- thus when $\eta < 0$ it follows that $\mu\leq \mathbbm{1}_{[-1,1]}$ and when $\eta > 0$ it follows that $\mu\geq \mathbbm{1}_{[-1,1]}$. We may of course choose $\mu$ to be the standard bump function with these properties, in which case the $C^k$ norms of $\mu$ satisfy $||\mu||_{C^k(\R)}\ll_k \eta^{-k}$, but let us only specialize to this choice later.

Let $\alpha\in C_c^\infty(G(\R))$ and $\beta\in C_c^\infty(L)$ be compactly supported smooth functions such that: $\alpha$ is $\SO_2(\R)$-invariant, $\int_{G(\R)} \alpha = 1$, $\beta(F_{\pm}) = 1$, and $\supp{\beta}, \beta^{-1}(\{1\})\subseteq L^{(1)}$ are both unions of two small compact intervals respectively containing $F_\pm$. Let $G_0 := \supp{\beta}$. Let $$\phi(v) := \sum_{g\cdot v_L = v} \alpha(g)\cdot \beta(v_L).$$

Write $V(\Z)^{\Delta\neq 0} := \{F\in V(\Z) : \Delta(F)\neq 0\}$ and $$V(\Z)^{\mathrm{nontriv.}} := \{F\in V(\Z)^{\Delta\neq 0} : 0\not\in F(\P^1(\Q))\}.$$ (Here we deviate from Ruth, and indeed Bhargava \cite{chapterthree-bhargava-thesis} and Bhargava-Shankar \cite{chaptertwo-bhargava-shankar-2-selmer}, in using the superscript $\mathrm{nontriv.}$ instead of $\mathrm{irred.}$, since in our view it is misleading to call these irreducible.) That is, $V(\Z)^{\mathrm{nontriv.}}$ is the subset of $V(\Z)$ consisting of binary quartic forms with no root in $\P^1(\Q)$ --- i.e., those binary quartics that do not have a linear factor defined over $\Q$.

Write $$Y(\Z) := \{F\in V(\Z) : I(F) = 0, J(F)\neq 0\}$$ and $Y(\Z)^\nontriv := Y(\Z)\cap V(\Z)^\nontriv$. We note that our $Y(\Z)$ and $Y(\Z)^\nontriv$ play the role of Ruth's $Y$ and $Y^{\mathrm{irr.}}$, respectively. Similarly, for $M\in \Z^+$ and $F_0\in V(\Z/M)$, write $$Y_{F_0\pmods{M}}(\Z) := \{F\in Y(\Z) : F\equiv F_0\pmods{M}\}$$ and $Y_{F_0\pmods{M}}(\Z)^\nontriv := Y_{F_0\pmods{M}}(\Z)\cap V(\Z)^\nontriv$.

Write, for $S\subseteq V(\Z)$, $$\#_{\phi,\mu}|B(u, t, \lambda, X)\cap S| := \sum_{F\in S} \mu\left(\frac{J(F)}{X}\right)\cdot \phi(a_t^{-1}\cdot n_u^{-1}\cdot (\lambda\cdot \id)^{-1}\cdot F),$$ where we have disambiguated the action of $\lambda$ (which should really be the action of $\smalltwobytwo{\lambda}{0}{0}{\lambda}$) by writing $\lambda\cdot \id$ for clarity. We note that this is different from Ruth's (and indeed Bhargava's) notation precisely because we have smoothed using $\alpha$, $\beta$, and $\mu$ rather than simply $\mathbbm{1}_{G_0}$ --- though the reader may well imagine that $\phi\sim \mathbbm{1}_{G_0\cdot L}$, in which case $\#_{\phi,\mu}|B(u, t, \lambda, X)\cap S|\sim \#|\lambda\cdot n_u\cdot a_t\cdot G_0\cdot L\cap S|$.

Just as in Section $2.2$ of Ruth's \cite{chapterthree-ruth's-thesis}, we find that the problem reduces to controlling $\#_{\phi,\mu}|B(u, t, \lambda, X)\cap Y_{F_0\pmods{M}}(\Z)^\nontriv|$. We do this with the following two lemmas.\footnote{Note that Lemma \ref{the bulk estimate} matches Ruth's Proposition $2.3.1$ in form, except that we have included a congruence condition in order to sieve to locally soluble forms --- Ruth overlooks doing this in his circle method analysis. We note also that Ruth overlooks a factor of the form $\sigma_\infty(u, t, \lambda, X)$ (specifically he overlooks the dependence on both $\lambda$ and $X$: on page $12$ of \cite{chapterthree-ruth's-thesis} he states that the condition $|J(v)| < X$ is superfluous once $\lambda\ll X^{\frac{1}{24}}$, thus one can drop it --- this is false, since his definition of his $\mathcal{F}'$ depends on a parameter $C'$ and were this to be the case the Selmer average would also grow with $C'$, rather than be $3$ --- one can think of our smoothing as solving this issue. Let us however emphasize here, lest our phrasing suggest otherwise, that such oversights are not out of the ordinary in original works like Ruth's, and that we hope we have made clear the accuracy of the idea contained therein.).}

\begin{lem}[The "tail" estimate.]\label{the tail estimate}
Let $$\lambda\in \R^+, u\in \left[-\frac{1}{2}, \frac{1}{2}\right], \sqrt{\frac{\sqrt{3}}{2}}\leq t\ll \lambda.$$ Then: $$\#_{\phi,\mu}|B(u, t, \lambda, X)\cap Y(\Z)^\nontriv|\ll_\phi \lambda^{12 + o(1)}.$$
\end{lem}

\begin{lem}[The "bulk" estimate.]\label{the bulk estimate}
Let $M\in \Z^+$ and $F_0\in V(\Z/M)$.
Let $$\lambda\in \R^+, u\in \left[-\frac{1}{2}, \frac{1}{2}\right], \sqrt{\frac{\sqrt{3}}{2}}\leq t\ll \lambda.$$ Then:
\begin{align*}
\#_{\phi,\mu}|B(u, t, \lambda, X)\cap Y_{F_0\pmods{M}}(\Z)| &= \sigma_\infty(u, t, \lambda, X)\cdot \prod_p  \sigma_p(Y_{F_0\pmods{M}}(\Z)) \\&\quad\quad + O_{\phi,M}(t^4\cdot \lambda^{8 + o(1)}) + O_{\phi, M, N}(||\mu||_{C^{O(N)}(\R)}\cdot t^N\cdot \lambda^{O(1) - N}),
\end{align*}\noindent
where $$\sigma_\infty(u, t, \lambda, X) := \lim_{\eps\to 0} \frac{\int_{v\in V(\R) : |I(v)|\leq \eps} dv\, \mu\left(\frac{J(v)}{X}\right)\cdot \phi(a_t^{-1}\cdot n_u^{-1}\cdot (\lambda\cdot \id)^{-1}\cdot v)}{2\eps}$$ and $$\sigma_p(Y_{F_0\pmods{M}}(\Z)) := \lim_{k\to \infty} p^{-4k}\cdot \#|\{F\in V(\Z/p^k) : I(F)\equiv 0\pmods{p^k}, F\equiv F_0\pmods{M}\}|.$$
\end{lem}\noindent
We have written $\sigma_\infty(u, t, \lambda, X)$ despite the function being independent of $u$ and $t$ (via $v\mapsto n_u\cdot a_t\cdot v$) for notational convenience.

Just as in e.g.\ Theorems $2.12$ and $2.21$ of Bhargava-Shankar's \cite{chaptertwo-bhargava-shankar-2-selmer}, we must introduce a weight function $\uglyphi: V(\Z/M)\to \R$ (which will eventually be taken to be a majorant of, in their notation, $f\mapsto \frac{1}{m(f)}$) with $M$ highly divisible. The following weighted version of Lemma \ref{the bulk estimate} of course follows immediately from Lemma \ref{the bulk estimate}.

\begin{lem}\label{the bulk estimate, weighted}
Let $M\in \Z^+$ and $\uglyphi: V(\Z/M)\to \R$.
Let $$\lambda\in \R^+, u\in \left[-\frac{1}{2}, \frac{1}{2}\right], \sqrt{\frac{\sqrt{3}}{2}}\leq t\ll \lambda.$$ Then:
\begin{align*}
&\sum_{F_0\in V(\Z/M)} \uglyphi(F_0)\cdot \#_{\phi,\mu}|B(u, t, \lambda, X)\cap Y_{F_0\pmods{M}}(\Z)|
\\&= \sigma_\infty(u, t, \lambda, X)\cdot \prod_{p\nmid M}  \sigma_p(Y(\Z))\cdot \lim_{k\to \infty} M^{-4k}\cdot \sum_{F\in V(\Z/M^k): I(F)\equiv 0\pmods{M^k}} \uglyphi(F\pmods{M})\\&\quad\quad + O_{\phi,M}(||\uglyphi||_1\cdot t^4\cdot \lambda^{8 + o(1)}) + O_{\phi,M,N}(||\uglyphi||_1\cdot ||\mu||_{C^{O(N)}(\R)}\cdot t^N\cdot \lambda^{O(1) - N}).
\end{align*}
\end{lem}

We note that the $p$-adic local densities are of course exact analogues of the singular integral at infinity --- one thickens $p$-adically by forcing only $I(F)\equiv 0\pmods{p^k}$, and then one takes a limit. It is because of this thickening that we easily reduce the calculation of the constants to results of Bhargava-Shankar.

\subsection{Avoiding a calculation of local densities.}

Let us first deduce Theorem \ref{all six families} for $k\equiv 1\pmods{6}$ from Lemmas \ref{the tail estimate} and \ref{the bulk estimate}.

\begin{proof}[Proof of Theorem \ref{all six families} for $k\equiv 1\pmods{6}$ assuming Lemmas \ref{the tail estimate} and \ref{the bulk estimate}.]
We reduce immediately to the case of $\mathcal{B}$ defined by finitely many congruence conditions. Indeed, assume Theorem \ref{all six families} for $k\equiv 1\pmods{6}$ for nonempty subsets of $\Z - \{0\}$ defined by finitely many congruence conditions. Writing $\mathcal{B}_p$ for the closure of $\mathcal{B}$ in $\Z_p$, and $\mathcal{B}_{\leq T} := \{n\in \Z - \{0\} : \forall p\leq T, n\in \mathcal{B}_p\}$ (thus $\mathcal{B}\subseteq \mathcal{B}_{\leq T}$), we find that, assuming Theorem $3.1.1$ for sets defined by finitely many congruence conditions (and thus for $\mathcal{B}_{\leq T}$):
\begin{align*}
&\sum_{B\in \mathcal{B} : |B|\leq X} \#|\Sel_2(E_{0,B}/\Q)|
\\&\leq \sum_{B\in \mathcal{B}_{\leq T} : |B|\leq X} \#|\Sel_2(E_{0,B}/\Q)|
\\&\leq (3 + O_{\mathcal{B}, T}(o_{X\to \infty}(1)))\cdot \#|\{n\in \mathcal{B}_{\leq T} : |n|\leq X\}|
\\&= (3 + O_{\mathcal{B}, T}(o_{X\to \infty}(1)))\cdot \left(\frac{\#|\{n\in \mathcal{B}_{\leq T} : |n|\leq X\}|}{\#|\{n\in \mathcal{B} : |n|\leq X\}|}\right)\cdot \#|\{n\in \mathcal{B} : |n|\leq X\}|
\\&= \left(3 + O_{\mathcal{B}}(o_{T\to \infty}(1)) + O_{\mathcal{B},T}(o_{X\to \infty}(1))\right)\cdot \#|\{n\in \mathcal{B} : |n|\leq X\}|.
\end{align*}\noindent
Taking $X\to \infty$ and then $T\to \infty$ gives that $$\Avg_{n\in \mathcal{B} : |n|\leq X} \#|\Sel_2(E_{0,B}/\Q)|\leq 3 + O_{\mathcal{B}}(o_{X\to \infty}(1)),$$ as desired.

Thus without loss of generality $\mathcal{B}$ is defined by finitely many congruence conditions, i.e.\ it is of the form $\mathcal{B} = \{n\in \Z - \{0\} : n\equiv a\pmod{m}\}$ for $m\in \Z^+$ and $a\in \Z/m$.

We will appeal to Ruth's Lemma $2.2.3$ (proven in Section $4.3$ of his \cite{chapterthree-ruth's-thesis} using an analysis of the (monogenized) cubic resolvent ring arising from a binary quartic form) to use the equality $n(F) = m(F)$ outside a negligible set. Here we use the notation of Section $3.2$ of Bhargava-Shankar's \cite{chaptertwo-bhargava-shankar-2-selmer}: $n(F)$ is the number of $\PGL_2(\Z)$-orbits in the (intersection of $V(\Z)$ and the) $\PGL_2(\Q)$-orbit of $F\in V(\Z)^{\Delta\neq 0}$, and
\begin{align*}
m(F) &:= \sum_{F'\in \PGL_2(\Z)\backslash (V(\Z)\cap \PGL_2(\Q)\cdot F)} \frac{\#|\Aut_{\PGL_2(\Q)}(F')|}{\#|\Aut_{\PGL_2(\Z)}(F')|}
\\&= \prod_p \sum_{F'\in \PGL_2(\Z_p)\backslash (V(\Z_p)\cap \PGL_2(\Q_p)\cdot F)} \frac{\#|\Aut_{\PGL_2(\Q_p)}(F')|}{\#|\Aut_{\PGL_2(\Z_p)}(F')|}
\\&=: \prod_p m_p(F).
\end{align*}

Let $M\in \Z^+$ with $m\vert M$. Let $\uglyphi: V(\Z/M)\to \R$. We will eventually take a sequence of majorants $\uglyphi_n$ of $F\mapsto \frac{1}{m(F)}$ --- i.e.\  $\uglyphi_1(F)\geq \cdots\geq \uglyphi_n(F)\geq \cdots\geq \frac{1}{m(F)}$ for all $F\in V(\Z)^{\Delta\neq 0}$ --- and apply the below to $\uglyphi_n$ and take $n\to \infty$.

As we have seen in e.g.\ Section \ref{smoothing section},
\begin{align*}
&N_\mu(Y_{F_0\pmods{M}}(\Z)^\nontriv, X)
\\&= \int_{\mathcal{F}} dh\, \sum_{v\in V(\Z)^\nontriv : I(v) = 0, F\equiv F_0\pmods{M}} \mu\left(\frac{J(v)}{X}\right)\cdot \phi(h^{-1}\cdot v)
\\&= \int_{1\ll \lambda\ll X^{\frac{1}{24}}} d^\times \lambda \int_{u\in \nu(t)} du \int_{\sqrt{\frac{\sqrt{3}}{2}}\leq t\ll \lambda} t^{-2} d^\times t\, \#_{\phi,\mu}|B(u, t, \lambda, X)\cap Y_{F_0\pmods{M}}(\Z)^\nontriv|
\\&= \int_{1\ll \lambda\ll X^{\frac{1}{24}}} d^\times \lambda \int_{u\in \nu(t)} du \int_{\sqrt{\frac{\sqrt{3}}{2}}\leq t\ll \lambda^\delta} t^{-2} d^\times t\, \#_{\phi,\mu}|B(u, t, \lambda, X)\cap Y_{F_0\pmods{M}}(\Z)^\nontriv|
\\&\quad + \int_{1\ll \lambda\ll X^{\frac{1}{24}}} d^\times \lambda \int_{|u|\leq \frac{1}{2}} du \int_{\lambda^\delta\ll t\ll \lambda} t^{-2} d^\times t\, \#_{\phi,\mu}|B(u, t, \lambda, X)\cap Y_{F_0\pmods{M}}(\Z)^\nontriv|,
\\&= \int_{1\ll \lambda\ll X^{\frac{1}{24}}} d^\times \lambda \int_{u\in \nu(t)} du \int_{\sqrt{\frac{\sqrt{3}}{2}}\leq t\ll \lambda^\delta} t^{-2} d^\times t\, \begin{pmatrix} \sigma_\infty(u, t, \lambda, X)\cdot \prod_p \sigma_p(Y_{F_0\pmods{M}}(\Z)) \\+\, O_{\phi,M}(||\mu||_{C^{O(1)}(\R)}\cdot t^4\cdot \lambda^{8 + o(1)})\end{pmatrix}
\\&\quad + \int_{1\ll \lambda\ll X^{\frac{1}{24}}} d^\times \lambda \int_{|u|\leq \frac{1}{2}} du \int_{\lambda^\delta\ll t\ll \lambda} t^{-2} d^\times t\,\, O_\phi(\lambda^{12 + o(1)})
\end{align*}\noindent
by definition, by splitting the integral, and then by Lemmas \ref{the tail estimate} and \ref{the bulk estimate}.

Next we pull out the product of local densities at finite primes in order to calculate the integral. We find that:
\begin{align*}
&N_\mu(Y_{F_0\pmods{M}}(\Z)^\nontriv, X)
\\&= \prod_p \sigma_p(Y_{F_0\pmods{M}}(\Z))\cdot \int_{1\ll \lambda\ll X^{\frac{1}{24}}} d^\times \lambda \int_{u\in \nu(t)} du \int_{\sqrt{\frac{\sqrt{3}}{2}}\leq t\ll \lambda^\delta} t^{-2} d^\times t\, \sigma_\infty(u, t, \lambda, X)
\\&\quad + O_{\phi,M}(||\mu||_{C^{O(1)}}(\R)\cdot X^{1 - \Omega(\delta) + o(1)})
\\&= \prod_p \sigma_p(Y_{F_0\pmods{M}}(\Z))\cdot \lim_{\eps\to 0} (2\eps)^{-1}\cdot \int_{\mathcal{F}} d^\times \lambda\, du\, t^{-2} d^\times t \int_{v\in V(\R) : |I(v)|\leq \eps} dv\, \mu\left(\frac{J(v)}{X}\right)\cdot \phi((\lambda\cdot a_t\cdot n_u)^{-1}\cdot v)
\\&\quad + O_{\phi,M}(||\mu||_{C^{O(1)}}(\R)\cdot X^{1 - \Omega(\delta) + o(1)})
\\&= \prod_p \sigma_p(Y_{F_0\pmods{M}}(\Z))\cdot \lim_{\eps\to 0} (2\eps)^{-1}\cdot \int_{v\in V(\R) : |I(v)|\leq \eps} dv\, \mu\left(\frac{J(v)}{X}\right) \int_{h\in \mathcal{F}} dh\, \phi(h^{-1}\cdot v)
\\&\quad + O_{\phi,M}(||\mu||_{C^{O(1)}}(\R)\cdot X^{1 - \Omega(\delta) + o(1)}),
\end{align*}\noindent
by definition of $\sigma_\infty$ and then by letting $h := (\lambda\cdot \id)\cdot a_t\cdot n_u$ and recalling that $\phi$ is $\SO_2(\R)$-invariant.

Recall that we saw in Section \ref{smoothing section} that, for $v\in V(\R)^{\Delta\neq 0}$ with $I(v) = 0$ (recall that then $\beta(v_L) = 1$), $$\int_{h\in \mathcal{F}} dh\, \phi(h^{-1}\cdot v) = \int_{G(\R)} dh\, \alpha(h)\cdot \#|\{g\in \mathcal{F} : gh\cdot v_L = v\}|.$$ Because we arranged that $\beta = 1$ on sufficiently small intervals around $F_\pm$, the same identity holds for $v\in V(\R)^{\Delta\neq 0}$ with $|I(v)|\leq \eps$ and $|J(v)|\gg 1$ when $\eps$ is sufficiently small. Inserting this into the above, we find that the main term is:
\begin{align*}
&(1 + O_M(X^{-1}))\cdot N_\mu(Y_{F_0\pmods{M}}(\Z)^\nontriv, X)
\\&= \prod_p \sigma_p(Y_{F_0\pmods{M}}(\Z))\\&\quad\quad\cdot \lim_{\eps\to 0} (2\eps)^{-1}\cdot \int_{v\in V(\R) : |I(v)|\leq \eps, |J(v)|\gg 1} dv\, \mu\left(\frac{J(v)}{X}\right) \int_{G(\R)} dh\, \alpha(h)\cdot \#|\{g\in \mathcal{F} : gh\cdot v_L = v\}|
\\&= \prod_p \sigma_p(Y_{F_0\pmods{M}}(\Z))\\&\quad\quad\cdot \lim_{\eps\to 0} (2\eps)^{-1}\cdot \int_{G(\R)} dh\, \alpha(h) \int_{v\in V(\R) : |I(v)|\leq \eps, |J(v)|\gg 1} dv\, \mu\left(\frac{J(v)}{X}\right) \#|\{g\in \mathcal{F} : gh\cdot v_L = v\}|
\\&= \prod_p \sigma_p(Y_{F_0\pmods{M}}(\Z))\cdot \lim_{\eps\to 0} (2\eps)^{-1}\cdot \int_{G(\R)} dh\, \alpha(h) \int_{v\in \mathcal{F}\cdot h\cdot L : |I(v)|\leq \eps, |J(v)|\gg 1} dv\, \mu\left(\frac{J(v)}{X}\right),
\end{align*}\noindent
by switching integrals and then noting that $\mathcal{F}\cdot h$ is a fundamental domain for the action of $G(\Z)\actson G(\R)$ (and using $\int_{G(\R)} \alpha = 1$). Note also that the factor $(1 + O_M(X^{-1}))$ in the first line comes from the fact that we have thrown out the $v\in V(\R)^{\Delta\neq 0}$ with $|I(v)|\leq \eps$ and $|J(v)|\ll 1$ in order to use the equality $\int_{h\in \mathcal{F}} dh\, \phi(h^{-1}\cdot v) = \int_{G(\R)} dh\, \alpha(h)\cdot \#|\{g\in \mathcal{F} : gh\cdot v_L = v\}|$.

Now the inner integral is independent of the choice of fundamental domain, whence:
\begin{align*}
&(1 + O_M(X^{-1}))\cdot N_\mu(Y_{F_0\pmods{M}}(\Z)^\nontriv, X)
\\&= \prod_p \sigma_p(Y_{F_0\pmods{M}}(\Z))\cdot \lim_{\eps\to 0} (2\eps)^{-1}\cdot \int_{v\in \mathcal{F}\cdot L : |I(v)|\leq \eps, |J(v)|\gg 1} dv\, \mu\left(\frac{J(v)}{X}\right)
\\&= 2X\cdot \prod_p \sigma_p(Y_{F_0\pmods{M}}(\Z))\cdot \lim_{\eps\to 0} (4\eps\cdot X)^{-1}\cdot \int_{v\in \mathcal{F}\cdot L : |I(v)|\leq \eps, |J(v)|\gg 1} dv\, \mu\left(\frac{J(v)}{X}\right)
\\&= \frac{2X}{m}\cdot \left(m\cdot\prod_p \sigma_p(Y_{F_0\pmods{M}}(\Z))\right)\cdot \lim_{\eps\to 0} (4\eps\cdot X)^{-1}\cdot \int_{v\in \mathcal{F}\cdot L : |I(v)|\leq \eps, |J(v)|\gg 1} dv\, \mu\left(\frac{J(v)}{X}\right),
\end{align*}\noindent

Now let us evaluate the remaining integral. By Proposition $2.8$ of Bhargava-Shankar's \cite{chaptertwo-bhargava-shankar-2-selmer} (we could avoid using this to get from the leftmost to the rightmost terms in the equality, of course), we have that
\begin{align*}
&(4\eps\cdot X)^{-1}\cdot \int_{v\in \mathcal{F}\cdot L : |I(v)|\leq \eps, |J(v)|\gg 1} dv\, \mu\left(\frac{J(v)}{X}\right)
\\&= (1 + O(\eta))\cdot (4\eps\cdot X)^{-1}\cdot \int_{v\in \mathcal{F}\cdot L : |I(v)|\leq \eps, 1\ll |J(v)|\leq (1 + O(\eta))\cdot X} dv
\\&= \frac{2\zeta(2)}{27}\cdot (1 + O(\eta + \eps\cdot X^{-1}))
\\&= (1 + O(\eta + \eps\cdot X^{-1}))\cdot \frac{\int_{\coprod_i \mathcal{R}^{(i)}(X)} dv}{\int_{\coprod_i R^{(i)}(X)} dI dJ},
\end{align*}\noindent
using the notation of Section $2.4$ of Bhargava-Shankar's \cite{chaptertwo-bhargava-shankar-2-selmer}. This concludes the first step in the trick of reducing the calculation of the product of local densities to the one done in Bhargava-Shankar --- at the moment, we have only dealt with the Archimedean restrictions.

Thus so far we have found that:
$$N_\mu(Y_{F_0\pmods{M}}(\Z)^\nontriv, X) = (1 + O_M(X^{-1}))\cdot \frac{2X}{m}\cdot \left(m\cdot \prod_p \sigma_p(Y_{F_0\pmods{M}}(\Z))\right)\cdot \frac{\int_{\coprod_i \mathcal{R}^{(i)}(X^2/4)} dv}{\int_{\coprod_i R^{(i)}(X^2/4)} dI dJ}.$$ Note that $$\#|\{J\in \mathcal{B} : 0\neq |J|\leq X\}| = \frac{2X}{m}\cdot \left(1 + O\left(\frac{m}{X}\right)\right).$$

Just as in the passage from Lemma \ref{the bulk estimate} to Lemma \ref{the bulk estimate, weighted}, we find that:
\begin{align*}
&\sum_{F_0\in V(\Z/M) : J(F_0)\equiv a\pmods{m}} \uglyphi(F_0)\cdot N_\mu(Y_{F_0\pmods{M}}(\Z)^\nontriv, X)
\\&= (1 + O(\eta) + O_M(X^{-1}))\cdot \frac{2X}{m}\cdot \frac{\int_{\coprod_i \mathcal{R}^{(i)}(X^2/4)} dv}{\int_{\coprod_i R^{(i)}(X^2/4)} dI dJ}\\&\quad\quad\cdot \left(m\cdot \lim_{T\to \infty} n_T^{-4}\cdot \sum_{F\in V(\Z/n_T) : I(F)\equiv 0\pmods{n_T}, J(F)\equiv a\pmods{m}} \uglyphi(F\pmods{M})\right),
\end{align*}\noindent
where $n_T := \prod_{p\leq T} p^T$, say, and without loss of generality $T\geq M$ so that $M\vert n_T$.

Choose now $\eta := X^{-\delta^2}$ and $\mu$ to be the standard bump function with the given properties, whence $||\mu||_{C^{O(1)}(\R)}\ll X^{O(\delta^2)}$.

Write $\mathbbm{1}_{\mathrm{sol.}}^{(p)}$ for the indicator function of the $\Q_p$-soluble binary quartic forms $F\in V(\Z_p)$ (that is to say, those $F$ for which $Z^2 = F(X,Y)$ admits a nonzero solution with all coordinates in $\Q_p$). Write $\mathbbm{1}_{\mathrm{loc.\ sol.}} := \prod_p \mathbbm{1}_{\mathrm{sol.}}^{(p)}.$ Write $$\uglyphi_*(F) := \frac{\mathbbm{1}_{\mathrm{loc.\ sol.}}}{m(F)} = \prod_p \frac{\mathbbm{1}_{\mathrm{sol.}}^{(p)}}{m_p(F)} =: \prod_p \uglyphi_*^{(p)}(F)$$ on $V(\Z)^{\Delta\neq 0}$. Thus when $1 =: \uglyphi_0^{(p)}(F)\geq \uglyphi_1^{(p)}(F)\geq \cdots\geq \uglyphi_n^{(p)}(F)\geq \cdots \geq \uglyphi_*^{(p)}(F)$ for all $F\in V(\Z_p)^{\Delta\neq 0}$ with $\uglyphi_n^{(p)}: V(\Z_p)^{\Delta\neq 0}\to [0,1]$ factoring through $V(\Z/p^n)$ (and not $V(\Z/p^{n-1})$) and such that $\uglyphi_n^{(p)}(F)\to \uglyphi_*^{(p)}(F)$ as $n\to \infty$, we find that\footnote{Of course one always has that, for a convergent sequence $x_k\in \R$, $\lim_{k\to \infty} x_k = x_k + o_{k\to \infty}(1)$, but here we have written $\lim_{k\to \infty} x_k = (1 + o_{k\to \infty}(1))\cdot x_k$, which is only justified (for $k$ sufficiently large) when $\lim_{k\to \infty} x_k\neq 0$. While we will see that the relevant limit is $2$, technically we are not yet justified in doing this and should carry the various additive error terms $o_{k\to \infty}(1)$ through the argument. However we hope the reader will grant us this notational simplification, since it makes no difference to the argument.}, writing $$\uglyphi_n := \prod_{p\leq n} \uglyphi_n^{(p)},$$ (and here is where we use Lemma $2.2.3$ of Ruth's \cite{chapterthree-ruth's-thesis} to replace $m(F)$ by $n(F)$ for all but $\ll X^{\frac{5}{6}}$ forms):
\begin{align*}
&\Avg_{B\in \mathcal{B} : |B|\leq X} \#|\Sel_2(E_{0,B}/\Q) - \{0\}|\\&\leq (1 + O_T(o_{X\to \infty}(1)))\cdot (1 + o_{T\to \infty}(1))\\&\quad\quad\cdot \frac{\int_{\coprod_i \mathcal{R}^{(i)}(X^2/4)} dv}{\int_{\coprod_i R^{(i)}(X^2/4)} dI dJ}\cdot \left(m\cdot n_T^{-4}\cdot \sum_{F\in V(\Z/n_T) : I(F)\equiv 0\pmods{n_T}, J(F)\equiv a\pmods{m}} \uglyphi_T(F)\right).
\end{align*}

Note that we have "thickened" by changing the constraint $I(F) = 0$ to $I(F)\equiv 0\pmods{n_T}$. Accordingly, we write $$\Inv^{(T, \mathcal{B})} := \{(I,J)\in \Z^2 : I\equiv 0\pmods{n_T}, J\equiv a\pmods{m}, 4I^3 - J^2\neq 0\},$$ and $\Inv_p^{(T, \mathcal{B})}\subseteq \Z_p^2$ for its closure in $\Z_p^2$.

The trick is now to notice that, for all $k\in \Z^+$,
\begin{align*}
&\sum_{F\in V(\Z/n_T) : I(F)\equiv 0\pmods{n_T}, J(F)\equiv a\pmods{m}} \uglyphi_T(F)
\\&= n_T^{-5\cdot (k-1)}\cdot \sum_{F\in V(\Z/n_T^k) : I(F)\equiv 0\pmods{n_T}, J(F)\equiv a\pmods{m}} \uglyphi_T(F),
\end{align*}\noindent
so that:
\begin{align*}
&m\cdot n_T^{-4}\cdot \sum_{F\in V(\Z/n_T) : I(F)\equiv 0\pmods{n_T}, J(F)\equiv a\pmods{m}} \uglyphi_T(F)
\\&= \frac{n_T^{-5k}\cdot \sum_{F\in V(\Z/n_T^k) : I(F)\equiv 0\pmods{n_T}, J(F)\equiv a \pmods{m}} \uglyphi_T(F)}{n_T^{-2k}\cdot \#|\{(I,J)\in V(\Z/n_T^k) : I\equiv 0\pmods{n_T}, J\equiv a\pmods{m}\}|}.
\end{align*}

Now we note that
\begin{align*}
&\lim_{k\to \infty} n_T^{-5k}\cdot \sum_{F\in V(\Z/n_T^k) : I(F)\equiv 0\pmods{n_T}, J(F)\equiv a\pmods{m}} \uglyphi_T(F) \\&= \int_{F\in V\left(\prod_{p\leq T} \Z_p\right) : (I(F), J(F))\in \prod_{p\leq T} \Inv_p^{(T, \mathcal{B})}} dF\, \uglyphi_T(F),
\end{align*}\noindent
and so, by dominated convergence and Fubini (note that we are implicitly using Proposition $3.18$ of Bhargava-Shankar's \cite{chaptertwo-bhargava-shankar-2-selmer}, and indeed arguing as in their proof of Proposition $2.21$ of their \cite{chaptertwo-bhargava-shankar-2-selmer}), it follows that\footnote{See the previous footnote (about factoring out $(1 + o_{k\to \infty}(1))$ and $(1 + o_{T\to \infty}(1))$).}
\begin{align*}
&n_T^{-5k}\cdot \sum_{F\in V(\Z/n_T) : I(F)\equiv 0\pmods{n_T}, J(F)\equiv a\pmods{m}} \uglyphi_T(F)
\\&= (1 + o_{T\to \infty}(1))\cdot (1 + O_T(o_{k\to \infty}(1)))
\\&\quad\quad\cdot \prod_p \int_{v\in V(\Z_p) : (I(v), J(v))\in \Inv_p^{(T, \mathcal{B})}} dv\, \uglyphi_*^{(p)}(v).
\end{align*}

Similarly, we of course have:
\begin{align*}
&n_T^{-2k}\cdot \#|\{(I,J)\in V(\Z/n_T^k) : I\equiv 0\pmods{n_T}, J\equiv a\pmods{m}\}|
\\&= \prod_{p\leq T} \int_{(I,J)\in \Inv_p^{(T, \mathcal{B})}} dI dJ
\\&= (1 + o_{T\to \infty}(1))\cdot \prod_p \int_{(I,J)\in \Inv_p^{(T, \mathcal{B})}} dI dJ.
\end{align*}

Combining all these, we find that:
\begin{align*}
&\Avg_{B\in \mathcal{B} : |B|\leq X} \#|\Sel_2(E_{0,B}/\Q) - \{0\}|\\&\leq (1 + O_T(o_{X\to \infty}(1)))\cdot (1 + o_{T\to \infty}(1))\cdot (1 + O_T(o_{k\to \infty}(1)))\\&\quad\quad\cdot \frac{\int_{\coprod_i \mathcal{R}^{(i)}(X^2/4)} dv}{\int_{\coprod_i R^{(i)}(X^2/4)} dI dJ}\cdot \frac{\prod_p \int_{v\in V(\Z_p) : (I(v), J(v))\in \Inv_p^{(T, \mathcal{B})}} dv\, \uglyphi_*^{(p)}(v)}{\prod_p \int_{(I,J)\in \Inv_p^{(T, \mathcal{B})}} dI dJ}.
\end{align*}

But the calculations in Section $3.6$ of Bhargava-Shankar's \cite{chaptertwo-bhargava-shankar-2-selmer} amount to the statement that $$\frac{\int_{\coprod_i \mathcal{R}^{(i)}(X^2/4)} dv}{\int_{\coprod_i R^{(i)}(X^2/4)} dI dJ}\cdot \frac{\prod_p \int_{v\in V(\Z_p) : (I(v), J(v))\in \Inv_p^{(T, \mathcal{B})}} dv\, \uglyphi_*^{(p)}(v)}{\prod_p \int_{(I,J)\in \Inv_p^{(T, \mathcal{B})}} dI dJ} = 2.$$ Taking $k\to \infty$, then $X\to \infty$, and finally $T\to \infty$, we deduce Theorem \ref{all six families} for $k\equiv 1\pmods{6}$.
\end{proof}

\subsection{The uniformity estimate.\label{the uniformity estimate for mordell curves section}}

In order to prove the matching lower bound (recall that we are in the special case where $\mathcal{B}\subseteq \Z - \{0\}$ is defined by finitely many congruence conditions) we simply run the above argument with minorants instead --- that is to say, we instead take $\eta := -X^{-\delta^2}$, $\mu$ again the standard bump function with the given properties, and $0 =: \uglyphi_0^{(p)}(F)\leq \uglyphi_1^{(p)}(F)\leq \cdots\leq \uglyphi_n^{(p)}(F)\leq \cdots\leq \uglyphi_*^{(p)}(F)$ for all $F\in V(\Z_p)^{\Delta\neq 0}$, with $\uglyphi_n: V(\Z_p)^{\Delta\neq 0}\to [0,1]$ factoring through $V(\Z/p^n)$ (and not $V(\Z/p^{n-1})$) and such that $\uglyphi_n^{(p)}(F)\to \uglyphi_*^{(p)}(F)$ as $n\to \infty$. The argument is precisely the same, with the exception of the first step.

Specifically, when proving the upper bound we implicitly used that the binary quartics $F\in V(\Z)^\nontriv$ with $I(F) = 0$ representing $2$-Selmer classes of $E_{0,J(F)}/\Q$ --- i.e.\ such that $Z^2 = F(X,Y)$ is nontrivially soluble in all completions of $\Q$ --- are in particular locally soluble at those $p\leq T$. However of course the converse does not hold. Similarly, in order to deduce a lower bound for the average of $\uglyphi_*$ we must have that $\uglyphi_T\leq \uglyphi_*$ outside a set of negligible size. So, just as in Bhargava-Shankar's \cite{chaptertwo-bhargava-shankar-2-selmer}, we need only prove that the integral with respect to Haar measure of the number of binary quartics $F$ in $B(u,t,\lambda,X)\cap Y(\Z)$ for which there is a prime $p$ with $p > \Pi$ such that $F$ is \emph{not} locally soluble or $\PGL_2(\Z_p)\cdot F\neq V(\Z_p)\cap \PGL_2(\Q_p)\cdot F$ is $$\ll \frac{\lambda^{12 + o(1)}}{\Pi\log{\Pi}}$$ when e.g.\ $\Pi\leq \lambda^{10^{-10}}$ and we are in the "bulk", so that $t\ll \lambda^{o(1)}$.

Now, just as in the proof of Proposition $3.18$ of Bhargava-Shankar's \cite{chaptertwo-bhargava-shankar-2-selmer}, if $F\in V(\Z)^\nontriv$ is a binary quartic that is not locally soluble at $p$, then $F\pmods{p}$ has splitting type one of $(1^2 1^2)$, $(2^2)$, or $(1^4)$. But if moreover $I(F) = 0$, and thus $I(F)\equiv 0\pmods{p}$, one gets much more: the splitting types $(1^2 1^2)$ and $(2^2)$ are not possible, as one can see by e.g.\ explicit calculation.\footnote{Working over $\Fbar_p$ and changing variables suitably, this amounts to the assertion that $$I(X^2\cdot (X - n\cdot Y)^2) = I(X^4 - 2n\cdot X^3 Y + n^2\cdot X^2 Y^2) = n^4.$$} Thus $F\pmods{p}$ must be a fourth power of a linear form, which is to say that $F\pmods{p}$ lies on the codimension $3$ subvariety $Z\subseteq V$ given by fourth powers of linear forms, namely the affine cone on a rational normal curve.

Similarly if $\PGL_2(\Z_p)\cdot F\neq V(\Z_p)\cap \PGL_2(\Q_p)\cdot F$, then, as in Bhargava-Shankar's \cite{chaptertwo-bhargava-shankar-2-selmer}, $p^2\mid \Delta(F)$, and indeed, up to $(X,Y)\mapsto (Y,X)$, there is a $\gamma\in \GL_2(\Z_p)$ such that $\diag(1,p)\cdot \gamma\cdot F\in V(\Z_p)$, which is to say that $(\gamma\cdot F)(X,Y) =: \widetilde{a}\cdot X^4 + \widetilde{b}\cdot X^3 Y + \widetilde{c}\cdot X^2 Y^2 + \widetilde{d}\cdot X Y^3 + \widetilde{e}\cdot Y^4$ has $p\mid \widetilde{d}$ and $p^2\mid \widetilde{e}$. Since $I(\gamma\cdot F) = 0$ this means that $p\mid \widetilde{c}$, and so using $I(\gamma\cdot F) = 0$ again we see that $p^2\mid \widetilde{b}\cdot \widetilde{d}$, whence either $p\mid \widetilde{b}$ and $\gamma\cdot F$ (and thus so too $F$) has splitting type $(1^4)$, or else $p^2\mid \widetilde{d}$ in which case $\diag(1,p)\cdot \gamma\cdot F$ has splitting type $(1^4)$, and of course the same holds after reduction via $\SL_2(\Z)$.

So it follows that in order to obtain the desired bound on the integral with respect to Haar measure of the number of $F\in B(u,t,\lambda,X)\cap Y(\Z)$ which are not locally soluble or have $\PGL_2(\Z_p)\cdot F\neq V(\Z_p)\cap \PGL_2(\Q_p)\cdot F$ at some $p$ with $p > \Pi$ it suffices to show that the number of $F\in B(u,t,\lambda,X)\cap Y(\Z)$ for which the reduction $F\pmods{p}\in Z(\F_p)$ for some $p > \Pi$ is $\ll \frac{\lambda^{12 + o(1)}}{\Pi\log{\Pi}}$, and then the desired bound follows from invoking Theorem $1.1$ of Browning-Heath-Brown's \cite{chapterthree-browning-heath-brown}.

\subsection{Point counting.}

Let us now prove Lemmas \ref{the tail estimate} and \ref{the bulk estimate}. We note that we will give an essentially one-line proof of Lemma \ref{the tail estimate} (namely, "use the divisor bound to determine $a,e$ from $b,c,d$"), which subsumes the entirety of Ruth's Section $3.5$ (pages $29-37$ of \cite{chapterthree-ruth's-thesis}).

\begin{proof}[Proof of Lemma \ref{the tail estimate}.]
Let $F\in \lambda\cdot n_u\cdot a_t\cdot G_0\cdot L\cap V(\Z)^\nontriv$. Write $F(X,Y) =: a\cdot X^4 + b\cdot X^3 Y + c\cdot X^2 Y^2 + d\cdot X Y^3 + e\cdot Y^4$. Evidently (by e.g.\ compactness of $G_0$ and $L$) we have that:
\begin{align*}
0\neq |a|&\ll \frac{\lambda^4}{t^4},
\\|b|&\ll \frac{\lambda^4}{t^2},
\\|c|&\ll \lambda^4,
\\|d|&\ll t^2\cdot \lambda^4,
\\0\neq |e|&\ll t^4\cdot \lambda^4.
\end{align*}\noindent
Therefore the number of tuples $(b,c,d)\in \Z^3$ among $F\in \lambda\cdot n_u\cdot a_t\cdot G_0\cdot L\cap V(\Z)^\nontriv$ is $\ll \lambda^{12}$.

Moreover, by hypothesis $12ae - 3bd + c^2 = 0$, i.e.\ $0\neq 12ae = 3bd - c^2\ll \lambda^8$. Thus $(b,c,d)$ determine $(a,e)$ up to $\ll \lambda^{o(1)}$ choices. In other words, the map $$\lambda\cdot n_u\cdot a_t\cdot G_0\cdot L\cap V(\Z)^\nontriv\to \Z^3$$ via $$a\cdot X^4 + b\cdot X^3 Y + c\cdot X^2 Y^2 + d\cdot X Y^3 + e\cdot Y^4\mapsto (b,c,d)$$ has image of size $\ll \lambda^{12}$ and fibres of size $\ll \lambda^{o(1)}$. The lemma follows.
\end{proof}

As for Lemma \ref{the bulk estimate}, we first note that it is essentially identical to Ruth's Proposition $3.4.1$ (modulo the inaccuracies in his treatment that we have already mentioned), which he states without proof.

For the reader's convenience we will give a full proof of Lemma \ref{the bulk estimate} anyway.
\begin{proof}[Proof of Lemma \ref{the bulk estimate}.]
Before we begin we note once again that it is not necessary to use the smoothed delta symbol method, since we are asking about zeroes of a quadric in five variables, something easily handled by the classical form of the circle method.

We follow the notation of Heath-Brown's \cite{chapterthree-heath-brown-delta-method}. Let $w_0\in C_c^\infty(\R)$ via $$w_0(x) := \begin{cases} \exp\left(-\frac{1}{(1-x^2)}\right) & |x| < 1\\ 0 & |x|\geq 1\end{cases}.$$ Let $\omega(x) := \frac{4}{\int_\R w_0(t) dt}\cdot w_0(x)$. Let $h(x,y) := \sum_{q\geq 1} \frac{\omega(qx) - \omega\left(\frac{|y|}{qx}\right)}{qx}$.

Note that $h(x,y) = 0$ when $x\gg 1 + |y|$ and that $h(x,y)\ll x^{-1}$.

We will first detail the argument in the case of $M = 1$ (i.e.\ no congruence condition) and then comment on necessary modifications to more general $M$ and $F_0\in V(\Z/M)$ as above.

Applying Theorem $2$ of Heath-Brown's \cite{chapterthree-heath-brown-delta-method} with his $n = 5$ and his $Q = \lambda^4$, we find that:
\begin{align*}
&\sum_{F\in V(\Z)} \phi(a_t^{-1}\cdot n_u^{-1}\cdot (\lambda\cdot \id)^{-1}\cdot F)
\\&= (\lambda^{-8} + O_N(\lambda^{-N}))\cdot \sum_{q\geq 1} q^{-5} \sum_{\vec{c}\in V(\Z)^*} \left(\sum_{u\in (\Z/q)^\times}\sum_{G\in V(\Z/q)} e_q(u\cdot I(G) + G\cdot \vec{c})\right)\\&\quad\quad\quad\quad\quad\quad\quad\quad\quad\quad\quad\quad\cdot \int_{F\in V(\R)} dF\, \phi(a_t^{-1}\cdot n_u^{-1}\cdot (\lambda\cdot \id)^{-1}\cdot F)\cdot h\left(\frac{q}{\lambda^4}, \frac{I(F)}{\lambda^8}\right)\cdot e_q(-F\cdot \vec{c}),
\end{align*}\noindent
where we have written $e_q(z) := e\left(\frac{z}{q}\right) := e^{\frac{2\pi i z}{q}}$.

For us the error term will consist of those terms where $\vec{c}\neq \vec{0}$, and the terms with $\vec{c} = \vec{0}$ will comprise the main term (in the end we will simply cite Heath-Brown's \cite{chapterthree-heath-brown-delta-method} for the analysis of the main term, which in any case is considerably simpler).

Via the change of variable $F\mapsto n_u\cdot a_t\cdot (\lambda\cdot \id)\cdot F$ (note that $(\lambda\cdot \id)\cdot F = \lambda^4\cdot F$ since $F$ is homogeneous of degree $4$ --- here $(\lambda\cdot \id)\cdot F$ on the left-hand side indicates the action of $\lambda\cdot \id\in G$ on $F\in V$ via $G\actson V$, and the $\cdot$ on the right-hand side denotes multiplication), we find that:
\begin{align*}
&\int_{F\in V(\R)} dF\, \phi(a_t^{-1}\cdot n_u^{-1}\cdot (\lambda\cdot \id)^{-1}\cdot F)\cdot h\left(\frac{q}{\lambda^4}, \frac{I(F)}{\lambda^8}\right)\cdot e_q(-F\cdot \vec{c})
\\&= \lambda^{20}\cdot \int_{F\in V(\R)} dF\, \phi(F)\cdot h\left(\frac{q}{\lambda^4}, I(F)\right)\cdot e_q(-\lambda^4 \cdot F\cdot ((n_u\cdot a_t)^\dag\cdot \vec{c})).
\end{align*}

Therefore we see that the error term is:
\begin{align*}
&(\lambda^{12} + O_N(\lambda^{-N}))\cdot \sum_{q\geq 1} q^{-5} \sum_{\vec{0}\neq \vec{c}\in V(\Z)^*} \left(\sum_{u\in (\Z/q)^\times}\sum_{G\in V(\Z/q)} e_q(u\cdot I(G) + G\cdot \vec{c})\right)\\&\quad\quad\quad\quad\quad\quad\quad\quad\quad\quad\quad\quad\quad\quad\cdot \int_{F\in V(\R)} dF\, \phi(F)\cdot h\left(\frac{q}{\lambda^4}, I(F)\right)\cdot e_q(-\lambda^4 \cdot F\cdot ((n_u\cdot a_t)^\dag\cdot \vec{c})).
\end{align*}

Note that the $\phi(F)$ term in the integral forces $||F||_\infty\ll_\phi 1$ if the integrand is to be nonzero, and then our observation that $h(x,y) = 0$ if $x\gg 1 + |y|$ forces $q\ll_\phi \lambda^4$ as well.

Note also that if $q\ll \lambda^{4-\eps}\cdot ||(n_u\cdot a_t)^\dag\cdot \vec{c}||_\infty$ --- i.e.\ if $$||(n_u\cdot a_t)^\dag\cdot \vec{c}||_\infty\gg \frac{q}{\lambda^{4-\eps}}$$ --- we find, by repeated integration by parts, that such terms contribute $O_{\eps,\phi,N}(t^N\cdot \lambda^{O(1) - N})$.

Also the complete exponential sum, which is just $$\sum_{u\in (\Z/q)^\times}\sum_{G_0, \ldots, G_4\in \Z/q} e_q(u\cdot (12 G_0 G_4 - 3 G_1 G_3 + G_2^2) + (c_0 G_0 + c_1 G_1 + c_2 G_2 + c_3 G_3 + c_4 G_4)),$$ is very easy to calculate (see e.g.\ page $49$ of \cite{my-thesis}). We conclude that $$\sum_{u\in (\Z/q)^\times}\sum_{G\in V(\Z/q)} e_q(u\cdot I(G) + G\cdot \vec{c})\ll \begin{cases} 0 & \exists p > 3 : v_p(q) = 1\\ q^{\frac{7}{2} + o(1)} & \forall p\vert q\text{ s.t.\ } p > 3, v_p(q)\geq 2,\end{cases}$$ and indeed one can sharpen the bound significantly.

Now we return to the smoothed delta symbol method. Again, the integral is nonnegligible only for $$||(n_u\cdot a_t)^\dag\cdot \vec{c}||_\infty\ll \frac{q}{\lambda^{4-\eps}}$$ (in which case it is $\ll_\phi \frac{\lambda^4}{q}$). Thus either $\vec{c} = \vec{0}$, in which case the corresponding summand contributes to the main term dealt with by Ruth (and by Heath-Brown in \cite{chapterthree-heath-brown-delta-method}), or else $\vec{c}\neq \vec{0}$, in which case we must have that $q\gg \frac{\lambda^{4-\eps}}{t^4}$ since by inspection $||(n_u\cdot a_t)^\dag\cdot \vec{c}||_\infty\gg t^{-4}\cdot ||\vec{c}||_\infty\gg t^{-4}$.

The error term is therefore:
\begin{align*}
&\ll \lambda^{12 + o(1)}\cdot \sum_{\frac{\lambda^{4-\eps}}{t^4}\ll q\ll \lambda^4,\, q\text{ powerful}} q^{-\frac{3}{2}} \\&\quad\quad\quad\quad\quad\quad \sum_{\vec{0}\neq \vec{c}\in V(\Z)^* : ||\vec{c}||_\infty\ll \frac{q}{\lambda^{4-\eps}}} \int_{F\in V(\R)} dF\, \phi(F)\cdot h\left(\frac{q}{\lambda^4}, I(F)\right)\cdot e_q(-\lambda^4 \cdot F\cdot ((n_u\cdot a_t)^\dag\cdot \vec{c})).
\end{align*}\noindent
Bounding the integrals trivially (i.e.\ by $\ll_\phi \frac{q}{\lambda^4}$) and noting that the number of $0\neq \vec{c}\in \Z^5$ such that $||(n_u\cdot a_t)^\dag\cdot \vec{c}||_\infty\ll \frac{q}{\lambda^{4-\eps}}$ is $$\ll \frac{t^4\cdot q}{\lambda^{4-\eps}}\cdot \left(1 + \frac{t^2\cdot q}{\lambda^{4-\eps}}\right)\cdot \left(1 + \frac{q}{\lambda^{4-\eps}}\right)\cdot \left(1 + \frac{q}{t^2\cdot \lambda^{4-\eps}}\right)\cdot \left(1 + \frac{q}{t^4\cdot \lambda^{4-\eps}}\right)$$ when $\frac{\lambda^{4-\eps}}{t^4}\ll q\ll \lambda^4$, we get that the error term is:
\begin{align*}
&\ll t^4\cdot \lambda^{4+\eps+o(1)}\cdot \sum_{\frac{\lambda^{4-\eps}}{t^4}\ll q\ll \lambda^4,\, q\text{ powerful}}  q^{\frac{1}{2}}\cdot \left(1 + \frac{t^2\cdot q}{\lambda^{4-\eps}}\right)\cdot \left(1 + \frac{q}{\lambda^{4-\eps}}\right)\cdot \left(1 + \frac{t^{-2}\cdot q}{\lambda^{4-\eps}}\right)\cdot \left(1 + \frac{t^{-4}\cdot q}{\lambda^{4-\eps}}\right)
\\&\ll t^6\cdot \lambda^{8 + 5\eps},
\end{align*}\noindent
as desired.

Thus we have bounded the error term suitably. The required analysis of the main term is already done in Heath-Brown's \cite{chapterthree-heath-brown-delta-method} (see e.g.\ the bottom of page $51$, i.e.\ the end of the proof of his Theorems $4$ and $5$), at least in the case $M = 1$.

We now discuss the modifications necessary for general $M\in \Z^+$ and $F_0\in V(\Z/M)$. First, in the application of the smoothed delta symbol method, instead of summing over $F\in V(\Z)$, we sum instead over the $F\in V(\Z)$ for which $F\equiv F_0\pmod{M}$ by summing over $\widetilde{F}\in V(\Z)$ and writing $F := F_0 + M\cdot \widetilde{F}$ (we implicitly choose a representative of $F_0$ in $V(\Z)$ and abuse notation by writing it as $F_0\in V(\Z)$). We then change variables from $\widetilde{F}$ back to $F$ in the integral and incur a factor of $M^{-5}$. The rest of the analysis of the error term is precisely the same (except that the primes one has to treat separately in the omitted complete exponential sum calculation are now those $p\vert 6M$, and the error terms now depend on $M$).

It remains to treat the main term, i.e.\ the local densities. We note that, by definition, we find local densities $$\sigma_p^{(F_0\pmods{M})}(Y(\Z)) := \lim_{k\to \infty} p^{-4k}\cdot \#|\{\widetilde{F}\in V(\Z/p^k) : I(F_0 + M\cdot \widetilde{F})\equiv 0\pmods{p^k}\}|.$$ We therefore are reduced to the claim that $$M^{-5}\cdot \prod_p \sigma_p^{(F_0\pmods{M})}(Y(\Z)) = \prod_p \sigma_p(Y_{F_0\pmods{M}}(\Z)).$$

Of course for $p\nmid M$ we have that $$\sigma_p^{(F_0\pmods{M})}(Y(\Z)) = \sigma_p(Y_{F_0\pmods{M}}(\Z)) = \sigma_p(Y(\Z)),$$ via the evident change of variables $\widetilde{F}\mapsto M^{-1}\cdot (\widetilde{F} - F_0)$.

However, it is also evident that $$M^{-5}\cdot \prod_{p\vert M} \sigma_p^{(F_0\pmods{M})}(Y(\Z)) = \prod_{p\vert M} \sigma_p(Y_{F_0\pmods{M}}(\Z)),$$ for the following reason. For $k\in \Z^+$ with $k\gg 1$ we have that: 
\begin{align*}
&M^{-5}\cdot \prod_{p\vert M} \sigma_p(Y_{F_0\pmods{M}}(\Z))
\\&= M^{-5}\cdot \prod_{p\vert M} \begin{pmatrix}
p^{-4k\cdot v_p(M)}\cdot \#|\{\widetilde{F}\in V(\Z/p^{k\cdot v_p(M)}) : I(F_0 + M\cdot \widetilde{F})\equiv 0\pmods{p^{k\cdot v_p(M)}}\}| \\\hspace{-3in}+ O(p^{-k\cdot v_p(M)})
\end{pmatrix}
\\&= \left(1 + O(e^{-\Omega_M(k)})\right)\cdot M^{-4k-5}\cdot \#|\{\widetilde{F}\in V(\Z/M^k) : I(F_0 + M\cdot \widetilde{F})\equiv 0\pmods{M^k}\}|,
\end{align*}\noindent
by the Chinese remainder theorem and that fact that all $p\geq 2$.

Because the condition $I(F_0 + M\cdot \widetilde{F})\equiv 0\pmods{M^k}$ only depends on $\widetilde{F}\pmods{M^{k-1}}$, we find that:
\begin{align*}
& M^{-4k-5}\cdot \#|\{\widetilde{F}\in V(\Z/M^k) : I(F_0 + M\cdot \widetilde{F})\equiv 0\pmods{M^k}\}|
\\&= M^{-4k}\cdot \#|\{F\in V(\Z/M^k) : I(F)\equiv 0\pmods{M^k}, F\equiv F_0\pmods{M}\}|
\\&= \prod_{p\vert M} \begin{pmatrix}
p^{-4k\cdot v_p(M)}\cdot \#|\{F\in V(\Z/p^{k\cdot v_p(M)}) : I(F)\equiv 0\pmods{p^{k\cdot v_p(M)}}, F\equiv F_0\pmods{M}\}| \\\hspace{-3in}+ O(p^{-k\cdot v_p(M)})
\end{pmatrix}.
\end{align*}

Thus taking $k\to \infty$ we find that $$M^{-5}\cdot \prod_{p\vert M} \sigma_p^{(F_0\pmods{M})}(Y(\Z)) = \prod_{p\vert M} \sigma_p(Y_{F_0\pmods{M}}(\Z)),$$ as desired.
\end{proof}

\section{Proof of Theorem \ref{all six families} for $k\equiv 2\pmods{6}$.}\label{square mordell curves section}

Now to the sketch of the proof of Theorem \ref{all six families} when $k\equiv 2\pmods{6}$.

\subsection{Reduction to point counting.}

We run an argument similar to the above, except our notation follows Bhargava-Ho's \cite{chapterthree-bhargava-ho-website} rather than Ruth's \cite{chapterthree-ruth's-thesis}, and we appeal in the end to part (f) of Theorem $1.1$ of Bhargava-Ho's \cite{chapterthree-bhargava-ho-website} (rather than Theorem $3.1$ of Bhargava-Shankar's \cite{chaptertwo-bhargava-shankar-2-selmer}) to calculate the product of local densities. Because the analytic part of the argument is essentially the same as in the previous section (in fact it is easier, since in the circle method argument we deal with a quadric in eight variables instead of five) we will be significantly more terse in this section.

Again, we follow the notation in Bhargava-Ho's \cite{chapterthree-bhargava-ho-website} (the relevant parametrization is by triply symmetric hypercubes). Let $V := 2\otimes \Sym_3(2)$, the space of pairs of "threes-in" binary cubic forms. Let $G$ be the image in $\GL(V)$ of $\{(g,h)\in \GL_2\times \GL_2 : \det{g}\cdot (\det{h})^3 = 1\}$, acting in the evident way on $2\otimes \Sym_3(2)$ (the first $\GL_2$ on the first factor via the standard representation, and the second $\GL_2$ on the second via the induced action on $\Sym_3$ of the standard representation). Note that $G\cong (\SL_2\times \SL_2)/\mu_2$.

We write, for $v\in V$, $$H(v) := \max\left(|I_2(v)|^{\frac{1}{2}}, |I_6(v)|^{\frac{1}{6}}\right)^{24},$$ where $I_2$ and $I_6$ are the invariants $a_2$ and $a_6$ of Section $6.3.2$ of Bhargava-Ho's \cite{chapterthree-bhargava-ho} and $a_1$ and $a_3$ of line $6$ (corresponding to the family $F_1(3)$) of Table $1$ in Bhargava-Ho's \cite{chapterthree-bhargava-ho-website}.

Let $R$ be a fundamental domain for $G(\R)\actson V(\R)^{\Delta\neq 0}$ (note that Bhargava-Ho write $V(\R)^{\mathrm{stab}} := V(\R)^{\Delta\neq 0}$), as constructed in Section $5$ of Bhargava-Ho's \cite{chapterthree-bhargava-ho-website} (via, in their notation, $R := \coprod_i R^{(i)}$). Let $L := \{v(\vec{a}) : \vec{a}\in (\R^m)^{\Delta\neq 0}_{H = 1}\}$ (here in their notation $m = 2$ and $v(\vec{a})$ is as defined in Section $4$ of Bhargava-Ho's \cite{chapterthree-bhargava-ho-website}) and $\Lambda := \{(\lambda\cdot \id, \id)\in \GL_2(\R)\times \GL_2(\R) : \lambda\in \R^+\}\subseteq \GL_2(\R)\times \GL_2(\R)$. Note that $R = \Lambda\cdot L$. Let $R(X) := \{v\in R : H(v)\leq X\}$. Let $\vec{F}_\pm := v((0,\pm 1))\in L$ be the two points in $L$ with $I_2 = 0$.

Note that, by construction, since $H((\lambda, \id)\cdot v) = \lambda^{24}\cdot H(v)$, the coefficients of a $v\in \lambda\cdot L\subseteq R(X)$ are all $\ll \lambda\ll X^{\frac{1}{24}}$, and hence, for $G_0\subseteq G(\R)$ compact, the coefficients of a $v\in \lambda\cdot G_0\cdot L\subseteq R(X)$ are all $\ll_{G_0} \lambda\ll X^{\frac{1}{24}}$.

Let $\mathcal{F}$ be a fundamental domain for $G(\Z)\actson G(\R)$, as constructed in Section $5.2$ of Bhargava-Ho's \cite{chapterthree-bhargava-ho-website}. Note that $\mathcal{F}$ lies inside the following Siegel set: $$\mathcal{F}\subseteq N\cdot A\cdot K,$$ where
\begin{align*}
N &:= \left\{(n_{u_1}, n_{u_2})\in G(\R) : |u_i|\leq \frac{1}{2}\right\},
\\A &:= \left\{(a_{t_1}, a_{t_2})\in G(\R) : t_i\geq \sqrt{\frac{\sqrt{3}}{2}}\right\},
\\K &:= \SO_2(\R)\times \SO_2(\R)\subseteq G(\R),
\end{align*}\noindent
with notation as before: $n_u := \smalltwobytwo{1}{0}{u}{1}$ and $a_t := \smalltwobytwo{t^{-1}}{0}{0}{t}$.

As before, let $\alpha\in C_c^\infty(G(\R))$ and $\beta\in C_c^\infty(L)$ be compactly supported smooth functions such that: $\alpha$ is $K$-invariant, $\int_{G(\R)} \alpha = 1$, $\beta(\vec{F}_{\pm}) = 1$, and $\supp{\beta}, \beta^{-1}(\{1\})\subseteq L^{(1)}$ are both unions of two small compact intervals respectively containing $\vec{F}_\pm$. Let $$\phi(v) := \sum_{g\cdot v_L = v} \alpha(g)\cdot \beta(v_L).$$

Let $V(\Z)^\nontriv := \{(F_1, F_2)\in V(\Z) : \forall [x,y]\in \P^1(\Q), \disc_{X,Y}(x\cdot F_1(X,Y) - y\cdot F_2(X,Y))\neq 0\}$. Let $Y(\Z) := \{v\in V(\Z) : I_2(v) = 0, I_6(v)\neq 0\}$ and $Y(\Z)^\nontriv := Y(\Z)\cap V(\Z)^\nontriv$.

For $M\in \Z^+$ and $v_0\in V(\Z/M)$, let $Y_{v_0\pmods{M}}(\Z) := \{v\in Y(\Z) : v\equiv v_0\pmods{M}\}$ and $Y_{v_0\pmods{M}}(\Z)^\nontriv := Y_{v_0\pmods{M}}(\Z)\cap V(\Z)^\nontriv$.

Write $n_{(u_1, u_2)} := (n_{u_1}, n_{u_2})$, and similarly $a_{(t_1, t_2)} := (a_{t_1}, a_{t_2})$.

Let, for $S\subseteq V(\Z)$, $$\#_{\phi,\mu}|B(\vec{u}, \vec{t}, \lambda, X)\cap S| := \sum_{\vec{F}\in S} \mu\left(\frac{I_6(\vec{F})}{X}\right)\cdot \phi(a_{\vec{t}}^{-1}\cdot n_{\vec{u}}^{-1}\cdot (\lambda\cdot \id, \id)\cdot (F_1, F_2)).$$

We again see that it suffices to prove the following two lemmas. Save for an easy (since one saves a nonzero constant probability at each prime) use of the Selberg upper bound sieve to control $\#_{\phi,\mu}|B(\vec{u}, \vec{t}, \lambda, X) - V(\Z)^\nontriv|$, the proof that these lemmas suffice, including reducing the evaluation of the resulting product of local densities by using the same trick to reduce to the same evaluation done (for the larger family $F_1(3)$) in the proof of part (f) of Theorem $1.1$ of Bhargava-Ho's \cite{chapterthree-bhargava-ho-website}, tracks that of the previous section, so we omit it --- though note that full details appear in \cite{sum-of-two-cubes-paper-with-ari-and-manjul}.

\begin{lem}\label{the tail estimate for squares}
Let $\lambda\in \R^+, u_i\in \left[-\frac{1}{2}, \frac{1}{2}\right], \sqrt{\frac{\sqrt{3}}{2}}\leq t_i\ll \lambda$. Then: $$\#_{\phi,\mu}|B(\vec{u}, \vec{t}, \lambda, X)\cap Y(\Z)^\nontriv|\ll_\phi \lambda^{6 + o(1)} + \lambda^{5 + o(1)}\cdot t_1\cdot t_2^3.$$
\end{lem}

Note that the second term is admissible because Haar measure on $G$ involves $t_1^{-2}\, d^\times t_1\, t_2^{-2}\, d^\times t_2$ and $t_1\cdot t_2\gg \lambda$ implies $B(\vec{u}, \vec{t}, \lambda, X)\cap Y(\Z)^\nontriv = \emptyset$.

\begin{lem}\label{the bulk estimate for squares}
Let $M\in \Z^+$ and $v_0\in V(\Z/M)$. Let $$\lambda\in \R^+, u_i\in \left[-\frac{1}{2}, \frac{1}{2}\right], \sqrt{\frac{\sqrt{3}}{2}}\leq t_i\ll \lambda.$$ Then:
\begin{align*}
\#_{\phi,\mu}|B(\vec{u}, \vec{t}, \lambda, X)\cap Y_{v_0\pmods{M}}(\Z)|
&= \sigma_\infty(\vec{u}, \vec{t}, \lambda, X)\cdot \prod_p \sigma_p(Y_{v_0\pmods{M}}(\Z)) \\&\quad\quad + O_\phi(||\vec{t}||_\infty^8\cdot \lambda^{4 + o(1)}) + O_{\phi, N}(||\vec{t}||_\infty^N\cdot \lambda^{O(1) - N}),
\end{align*}\noindent
where $$\sigma_\infty(\vec{u}, \vec{t}, \lambda, X) := \lim_{\eps\to 0} \frac{\int_{v\in V(\R) : |I_2(v)|\leq \eps} dv\, \mu\left(\frac{I_6(\vec{F})}{X}\right)\cdot \phi(a_{\vec{t}}^{-1}\cdot n_{\vec{u}}^{-1}\cdot (\lambda\cdot \id, \id)^{-1}\cdot v)}{2\eps}$$ and $$\sigma_p(Y_{v_0\pmods{M}}(\Z)) := \lim_{n\to \infty} p^{-4n}\cdot \#|\{\vec{F}\in V(\Z/p^n) : I_2(\vec{F})\equiv 0\pmods{p^n}, v\equiv v_0\pmods{M}\}.$$
\end{lem}

\subsection{The uniformity estimate.}\label{uniformity estimate for square mordell curves section}

In fact the proof of the first part of the uniformity estimate reduces to the one proven in Section \ref{the uniformity estimate for mordell curves section}, for the following reason. Recall that, given a pair of binary cubic forms $\vec{F} =: (F_1, F_2)$ with each $F_i\in \Z[X,Y]$, one produces a binary quartic form via $$G_{\vec{F}}(x,y) := \disc_{X,Y}(x\cdot F_1(X,Y) - y\cdot F_2(X,Y))\in \Z[x,y].$$ Note that, as one can see by e.g.\ explicit calculation, $I_2(\vec{F})\,\vert\, I(G_{\vec{F}})$.

Now in fact one has by definition that the pair $\vec{F} = (F_1, F_2)$ is locally soluble at $p$ if and only if $-\frac{1}{3}\cdot G_{\vec{F}}$ is locally soluble at $p$. Therefore to bound the number of $\vec{F}\in B(\vec{u}, \vec{t}, \lambda, X)\cap Y(\Z)$ which are not locally soluble at a $p$ with $p > \Pi$, it suffices to observe that the statement that $\vec{F}$ is not locally soluble at $p$ implies that $G_{\vec{F}}\pmods{p}$ lies on a codimension $3$ subvariety of the space of binary quartics, so that (after checking the independence of the resulting three equations in the coefficients of $\vec{F}$, which in fact imply that either $F_2$ is proportional to $F_1$ or else $F_1 = 0$) $\vec{F}\pmods{p}$ lies on a codimension $3$ subvariety of the space of pairs of binary cubics, in which case we may again apply Theorem $1.1$ of Browning-Heath-Brown's \cite{chapterthree-browning-heath-brown} to conclude.

The proof of the second part of the uniformity estimate --- that is, the bound on the number of $\vec{v}$ with $G(\Z_p)\cdot \vec{v}\neq G(\Q_p)\cdot \vec{v}\cap V(\Z_p)$ for some $p > \Pi$ --- is more intricate and appears in full in \cite{sum-of-two-cubes-paper-with-ari-and-manjul}.

\subsection{Point counting.}

The proof of Lemma \ref{the tail estimate for squares} is very much the same as the proof of Lemma \ref{the tail estimate}, except that nontriviality does not rule out points in a particular totally isotropic subspace, so that we must treat them separately.

\begin{proof}[Proof of Lemma \ref{the tail estimate for squares}.]
Recall that $t_1\cdot t_2\ll \lambda$ (else $B(\vec{u}, \vec{t}, \lambda, X)\cap Y(\Z)^\nontriv = \emptyset$). As mentioned, each coefficient of an $(F^{(1)}, F^{(2)}) =: \vec{F}\in B(\vec{u}, \vec{t}, \lambda, X)\cap Y(\Z)^\nontriv$ is $\ll_\phi \lambda$. If there is an $i\in \{0, 1, 2, 3\}$ such that $F^{(1)}_i\cdot F^{(2)}_{3-i}\neq 0$ then, applying the divisor bound, we determine $(F^{(1)}_i, F^{(2)}_{3-i})$, up to $\ll_\phi \lambda^{o(1)}$ choices, from $((F_k^{(1)}, F_{3-k}^{(2)}))_{0\leq k\leq 3, k\neq i}$, and there are $\ll_\phi \lambda^5\cdot t_1\cdot t_2^3\cdot \left(\frac{\lambda}{t_1\cdot t_2^3} + 1\right) = \lambda^6 + \lambda^5\cdot t_1\cdot t_2^3$ choices for the latter.

It remains to treat those $\vec{F}$ with $F^{(1)}_i\cdot F^{(2)}_{3-i} = 0$ for all $0\leq i\leq 3$. If there is a $1\leq j\leq 2$ and an $i\in \{0, 3\}$ such that $F^{(j)}_i = F^{(j)}_{i+(-1)^i} = 0$ or else $F^{(j)}_i = F^{(3-j)}_i = 0$ then evidently $\vec{F}\not\in V(\Z)^\nontriv$. Hence there is a $1\leq j\leq 2$ such that $F^{(j)}_0 = F^{(j)}_3 = F^{(3-j)}_1 = F^{(3-j)}_2 = 0$. Thus the number of such $\vec{F}\in B(\vec{u}, \vec{t}, \lambda, X)$ is $\ll \left(\frac{\lambda}{t_1\cdot t_2^3} + 1\right)\cdot \frac{\lambda\cdot t_2^3}{t_1}\cdot \frac{\lambda\cdot t_1}{t_2}\cdot \lambda\cdot t_1\cdot t_2 + \frac{\lambda}{t_1\cdot t_2}\cdot \frac{\lambda\cdot t_2}{t_1}\cdot \left(\frac{\lambda\cdot t_1}{t_2^3} + 1\right)\cdot \lambda\cdot t_1\cdot t_2^3 = \lambda^4 + \lambda^3\cdot t_1\cdot t_2^3$, as desired.
\end{proof}

The proof of Lemma \ref{the bulk estimate for squares} is also similar to the proof of Lemma \ref{the bulk estimate}. We comment on the two major differences --- again, full details are given in \cite{sum-of-two-cubes-paper-with-ari-and-manjul}.

\begin{proof}[Proof of Lemma \ref{the bulk estimate for squares}.]
The principal differences are the following. In applying the smoothed delta symbol method (i.e.\ Theorem $2$ of Heath-Brown's \cite{chapterthree-heath-brown-delta-method}), we take $n = 8$ and $Q = \lambda$. In calculating the complete exponential sum, we note that, because we have an even number of variables and thus the mod-$q$ complete exponential sums no longer vanish for $q$ prime, we instead use the bound $\ll q^{5 + o(1)}$ for all $q$. For the same reason we no longer reduce to a sum over only powerful $q$, but rather we sum over all $q\ll \lambda$ in the bounding of the error term.

Otherwise the analytic argument is mutatis mutandis the same.
\end{proof}

\section{Proof of Theorem \ref{all six families} for $k\equiv 5\pmods{6}$.}\label{quintic mordell curves section}

Now to the sketch of the proof of Theorem \ref{all six families} when $k\equiv 5\pmods{6}$. As noted we will deduce this case from the intermediate results used to deduce the case $k\equiv 2\pmods{6}$ as follows.

We begin with the following preliminary lemma.

\begin{lem}\label{bad points don't actually matter}
Let $p\gg 1$ be prime. Let $A,B\in \Z_p$ be such that $-4A^3 - 27B^2\neq 0$. Let $E_{A,B} : y^2 = x^3 + A x + B$ and let $\mathcal{E}_{A,B}/\Z_p$ be its scheme-theoretic closure in $\P^2/\Z_p$. Let $\Sigma\subseteq E_{A,B}(\F_p)$ be such that $\#|\Sigma|\ll 1$. Let $P\in E_{A,B}(\Q_p)$. Then: there is an $(x,y)\in \mathcal{E}_{A,B}(\Z_p)$ with $x,y\in \Z_p^\times$ and $(x,y)\pmods{p}\not\in \Sigma$ such that $(x,y)\equiv P\pmods{2\cdot E_{A,B}(\Q_p)}$ as elements of $E_{A,B}(\Q_p)$.
\end{lem}

\begin{proof}
Let $S := \Sigma\cup \{\infty\}\cup \{(x,y)\in \mathcal{E}_{A,B}(\F_p) : x\cdot y = 0\}\cup (\mathcal{E}_{A,B}\bmod{p})^{\mathrm{sing.}}\subseteq \mathcal{E}_{A,B}(\Fbar_p)$. Note that $\#|S|\ll 1$. Next note that $0\to \mathcal{E}_{A,B}(\F_p)[2]\to \mathcal{E}_{A,B}(\F_p)\to 2\cdot \mathcal{E}_{A,B}(\F_p)\to 0$, and $\#|\mathcal{E}_{A,B}(\F_p)[2]|\ll 1$, so that $\#|2\cdot \mathcal{E}_{A,B}(\F_p)|\gg p$ since $\#|\mathcal{E}_{A,B}(\F_p)| = p + O(\sqrt{p})\gg p$. Consequently for each $\overline{P}\in \mathcal{E}_{A,B}(\F_p)$ there is a $\overline{Q}\in \mathcal{E}_{A,B}(\F_p)$ such that $\overline{Q}, \overline{P} + 2\overline{Q}\not\in S$ (because $p\gg 1$).

So for each $P\in E_{A,B}(\Q_p)$ there is a $\overline{Q}\in \mathcal{E}_{A,B}(\F_p)$ such that $\overline{Q}, (P\bmod p) + 2\overline{Q}\not\in S$, whence by Hensel a $Q\in E_{A,B}(\Q_p)$ such that, writing $P + 2Q =: (x,y)$, one has $x,y\in \Z_p^\times$, as desired.
\end{proof}

Next let us note the following.

\begin{lem}\label{either a simple root or else being the fourth power of a linear form is forced mod bad primes}
Let $p\gg 1$. Let $F\in \Sym^4(2)(\Z)$ be such that $I(F) = 0$ and $p\mid J(F)$. Then: either $p\mid F$ or else $F$ has splitting type either $(1^4)$ or else $(1^3 1)$ modulo $p$.
\end{lem}

\begin{proof}
If $p\mid F$ we are done. Otherwise by hypothesis $F$ has a double root modulo $p$, so its possible splitting types are $(1^4), (1^3 1)$, $(1^2 1^2)$, $(2^2)$, and $(1^2 1 1)$. Because $I(X^2\cdot (\alpha\cdot X^2 + \beta\cdot X Y + \gamma\cdot Y^2)) = \gamma^2$ all but the first two splitting types are ruled out and we are done.
\end{proof}

This has the following consequence.

\begin{lem}\label{relevant local binary quartics have a root}
Let $p\gg 1$ be a prime. Let $F\in \Sym^4(2)(\Z)$ with $I(F) = 0$ and $p\mid J(F)$ be such that $p\nmid F$ and such that $F$ does not have splitting type $(1^4)$ modulo $p$. Then: $F$ has a root in $\P^1(\Q_p)$.
\end{lem}

\begin{proof}
By Lemma \ref{either a simple root or else being the fourth power of a linear form is forced mod bad primes} $F$ has a simple root over $\F_p$, QED by Hensel.
\end{proof}

We use Lemma \ref{relevant local binary quartics have a root} as follows.

\begin{lem}\label{at large primes relevant binary quartics are locally soluble even when twisted}
Let $p\gg 1$ be a prime. Let $F\in \Sym^4(2)(\Z)$ with $I(F) = 0$ be such that $Z^2 = F(X,Y)$ has a $\Q_p$-point. Let $d\mid J(F)$. Then: $d\cdot Z^2 = F(X,Y)$ has a $\Q_p$-point, and so in particular if $a^3\mid J(F)$ then it follows that there is a $g\in \PGL_2(\Q)$ such that $\frac{(g\cdot F)(X,Y)}{a}\in \Z_p[X,Y]$ with $(g\cdot F)(X,Y) := (\det{g})^{-2}\cdot F((X,Y)\cdot g)$.
\end{lem}

\begin{proof}
If $p\nmid J(F)$ we are done because $p\gg 1$. Otherwise because $Z^2 = F(X,Y)$ has a $\Q_p$-point $F$ is $\PGL_2(\Q_p)$-equivalent to the Kummer image $F_{(x,y)}(X,Y) := X^4 - 6x\cdot X^2 Y^2 + 8y\cdot X Y^3 - 3x^2\cdot Y^4$ of a point $(x,y)\in E_{0,-27\cdot J(F)}(\Q_p)$ with (by Lemma \ref{bad points don't actually matter}) $x,y\in \Z_p^\times$, which is thus automatically not of mod-$p$ splitting type $(1^4)$, and so by Lemma \ref{relevant local binary quartics have a root} we conclude that $F$ has a root in $\P^1(\Q_p)$. Thus indeed $Z^2 = \frac{F(X,Y)}{d}$ has a $\Q_p$-point (indeed one with $Z = 0$), giving the first claim. As for the second claim, it follows in that case that there is a $g\in \PGL_2(\Q_p)$ such that $\frac{(g\cdot F)(X,Y)}{a}$ is the Kummer image of a point $(x,y)\in E_{0,\frac{-27 J(F)}{a^3}}(\Q_p)$ with (by Lemma \ref{bad points don't actually matter}) $x,y\in \Z_p^\times$, and so in particular that $\frac{(g\cdot F)(X,Y)}{a}\in \Z_p[X,Y]$ as desired (to conclude approximate $g$ sufficiently closely by an element of $\PGL_2(\Q)$).
\end{proof}

Now let us repeat the notation of Section \ref{square mordell curves section}: let $V := 2\otimes \Sym_3(2)$, let $G$ be the image in $\GL(V)$ of $\{(g,h)\in \GL_2\times \GL_2 : \det{g}\cdot (\det{h})^3 = 1\}$, etc.

Note that the (first) invariant binary quartic associated by Bhargava-Ho in their \cite{chapterthree-bhargava-ho} to a $\vec{v}\in V$ --- via $V\inj 2^{\otimes 4}$, i.e.\ via regarding a pair of binary cubic forms as a triply symmetric hypercube in the usual way (i.e.\ $\left(\sum_{i=0}^3 {3\choose i}\cdot a_i^{(j)}\cdot x^i\right)_{j=1}^2\mapsto \left(a_{j+k+\ell-3}^{(i)}\right)_{1\leq i,j,k,\ell\leq 2}$) --- is $F_{\vec{v}}(X,Y) := -\frac{1}{27}\cdot \disc_{(x,y)}(X\cdot v_1(x,y) - Y\cdot v_2(x,y))$. By construction $F_{\left(\sum_{i=0}^3 {3\choose i}\cdot a_i^{(j)}\cdot x^i\right)_{j=1}^2}\in \Z\left[a_0^{(1)}, \ldots, a_3^{(2)}, X, Y\right]$.

\begin{lem}\label{integral representatives for quintic mordell curves}
There is an absolute constant $\kappa\in \Z - \{0\}$ such that the following holds. Let $0\neq B\in \Z$ be sixth-power-free. Then: there is a bijection between $\Sel_2(E_{0, B^5}/\Q) - \{0\}$ and $$G(\Q)\backslash \left\{\vec{v}\in \frac{1}{\kappa}\cdot V(\Z)^\nontriv : \substack{a_2(\vec{v}) = 0, a_6(\vec{v}) = 2 B, \\B\cdot Z^2 = F_{\vec{v}}(X,Y)\text{ everywhere locally soluble}}\right\}.$$
\end{lem}

\begin{proof}
Factor $B =: B_*\cdot B_\square^2$ with $B_*$ squarefree --- thus $E_{0,B^5}\simeq E_{0,B^2\cdot B_*^3}$ over $\Q$.

Of course (by e.g.\ Theorem $3.5$ of Bhargava-Shankar's \cite{chaptertwo-bhargava-shankar-2-selmer}) $\Sel_2(E_{0,B^2\cdot B_*^3}/\Q) - \{0\}$ is in bijection with $\PGL_2(\Q)\backslash \{F\in \Sym^4(2)(\Z)^\nontriv : I(F) = 0, J(F) = -2^6\cdot 3^3\cdot B^2\cdot B_*^3, Z^2 = F(X,Y)\text{ everywhere locally soluble}\}$. Given such an $F\in \Sym^4(2)(\Z)^\nontriv$, we first claim that there is an absolute constant $\kappa\in \Z - \{0\}$ (thus $|\kappa|\ll 1$), a $\vec{v}\in \frac{1}{\kappa}\cdot V(\Z)^\nontriv$, and a $g\in \PGL_2(\Q)$ such that $B_*\cdot F_{\vec{v}}(X,Y) = (g\cdot F)(X,Y)$.

To see this, note that the curve $B_*\cdot Z^2 = F(X,Y)$ is a genus one curve with a degree-two line bundle (via the pullback of $\mathcal{O}(1)$ via $[X,Y,Z]\mapsto X/Y$) whose Jacobian (namely $y^2 = x^3 + B^2$) has a marked rational $3$-torsion point (namely $(0,B)$). Thus by Theorem $6.16$ of Bhargava-Ho's \cite{chapterthree-bhargava-ho} it follows that there is a $\vec{v}\in V(\Q)$ such that $\frac{F(X,Y)}{B_*}$ is $\PGL_2(\Q)$-equivalent to $F_{\vec{v}}(X,Y)$.

Note that this already shows that, for each $p$, there is a $\vec{w}\in G(\Q_p)\cdot \vec{v}$ with $\vec{w}\in p^{-O(1)}\cdot V(\Z_p)$, since by following the procedure outlined at the bottom of page $57$ of Bhargava-Ho's \cite{chapterthree-bhargava-ho} (and symmetrizing suitably) one produces a universal formula for a $\widetilde{\vec{v}}\in V\left(\Z\left[\mathrm{coeff.s}(\widetilde{F}), \sqrt{-27 J(\widetilde{F})}, (6\cdot J(\widetilde{F}))^{-1}\right]\right)$ such that $F_{\widetilde{\vec{v}}}$ is $\PGL_2\left(\Q\left(\mathrm{coeff.s}(\widetilde{F}), \sqrt{-27 J(\widetilde{F})}\right)\right)$-equivalent to a given $\widetilde{F}\in \Sym^4(2)$ with $I(\widetilde{F}) = 0$ and $J(\widetilde{F})\neq 0$.

But furthermore Lemma \ref{at large primes relevant binary quartics are locally soluble even when twisted} implies that $\vec{v}$ is locally soluble at all $p\gg 1$, so that, just as in Bhargava-Ho's \cite{chapterthree-bhargava-ho}, by strong approximation for $G/\Q$ it follows that there is a $\gamma\in G(\Q)$ such that, for all $p$, $\gamma\cdot \vec{v}\in p^{-O(1)}\cdot V(\Z_p)$, and, for all $p\gg 1$, $\gamma\cdot \vec{v}$ is furthermore $G(\Z_p)$-equivalent to the Kummer image\footnote{Here to define the Kummer image we use the formula specified in the paragraph after $(14)$ on page $14$ of Bhargava-Ho's \cite{chapterthree-bhargava-ho-website} in this case (so crucially $v = w = 0$ in their notation).} of an $(x,y)\in E_{0,B^2}(\Q_p)$ with (by Lemma \ref{bad points don't actually matter}) $x,y\in \Z_p^\times$, so that in particular $\gamma\cdot \vec{v}\in V(\Z_p)$ for these $p$.

This proves the desired claim that for each locally soluble $F\in \Sym^4(2)(\Z)^\nontriv$ with $I(F) = 0$ and $J(F) = -2^6\cdot 3^3\cdot B^2\cdot B_*^3$ there is a $\vec{v}\in \frac{1}{\kappa}\cdot V(\Z)^\nontriv$ such that $F_{\vec{v}}(X,Y)$ and $\frac{F(X,Y)}{B_*}$ are $\PGL_2(\Q)$-equivalent. The map $\PGL_2(\Q)\cdot F\mapsto G(\Q)\cdot \vec{v}$ thus defined (i.e.\ by appeal to Theorem $6.16$ of Bhargava-Ho's \cite{chapterthree-bhargava-ho}) is certainly injective because we recover $\PGL_2(\Q)\cdot F$ from $G(\Q)\cdot \vec{v}$ by the defining property of the map, but it is also of course surjective --- given $\vec{v}$ just form $B_*\cdot F_{\vec{v}}(X,Y)$ --- so we are done.
\end{proof}

Now we may prove Theorem \ref{all six families} for $k\equiv 5\pmods{6}$.

\begin{proof}[Proof of Theorem \ref{all six families} for $k\equiv 5\pmods{6}$.]
Repeat the proof of Theorem \ref{all six families} in the case $k\equiv 2\pmods{6}$ --- including the appeal to the uniformity estimate which is proved in full in \cite{sum-of-two-cubes-paper-with-ari-and-manjul} --- mutatis mutandis, with the exception that the local conditions (besides those modulo $\kappa^{O(1)}$) imposed be the everywhere local solubility of $2 a_6(\vec{v})\cdot Z^2 = F_{\vec{v}}(X,Y)$ --- rather than that of $Z^2 = F_{\vec{v}}(X,Y)$ --- throughout. (Note that the fact that these provide the same local densities ultimately reduces to Lemma $3.20$ of Bhargava-Shankar's \cite{chaptertwo-bhargava-shankar-2-selmer}.)
\end{proof}

\section{Proof of Theorem \ref{all six families} for $k\equiv 4\pmods{6}$.}\label{quartic mordell curves section}

Now to the sketch of the proof of Theorem \ref{all six families} when $k\equiv 4\pmods{6}$. The relevant inputs are as follows.

\begin{lem}\label{integral representatives for quartic mordell curves}
There is an absolute constant $\kappa\in \Z - \{0\}$ such that the following holds. Let $0\neq B\in \Z$ be sixth-power-free. Then: there is a bijection between $\Sel_2(E_{0, B^4}/\Q) - \{0\}$ and $$\PGL_2(\Q)\backslash \left\{F\in \frac{1}{\kappa}\cdot V(\Z)^\nontriv : \substack{I(F) = 0, J(F) = -2^6\cdot 3^3\cdot B, \\B\cdot Z^2 = F(X,Y)\text{ everywhere locally soluble}}\right\}.$$
\end{lem}

\begin{proof}
Factor $B =: B_*\cdot B_\square^2$ with $B_*$ squarefree --- thus $E_{0,B^4}\simeq E_{0,B\cdot B_*^3}$ over $\Q$. By e.g.\ Theorem $3.5$ of Bhargava-Shankar's \cite{chaptertwo-bhargava-shankar-2-selmer} $\Sel_2(E_{0,B^4}/\Q) - \{0\}$ is in bijection with $$\PGL_2(\Q)\backslash \left\{F\in \Sym^4(2)(\Z)^\nontriv : \substack{I(F) = 0, J(F) = -2^6\cdot 3^3\cdot B\cdot B_*^3, \\Z^2 = F(X,Y)\text{ everywhere locally soluble}}\right\}.$$

Given such an $F\in \Sym^4(2)(\Z)^\nontriv$ we claim that there is an absolute constant $\kappa\in \Z - \{0\}$ (thus $|\kappa|\ll 1$) and a $g\in \PGL_2(\Q)$ such that $\frac{(g\cdot F)(X,Y)}{B_*}\in \frac{1}{\kappa}\cdot \Z[X,Y]$, where $(g\cdot F)(X,Y) := (\det{g})^{-2}\cdot F((X,Y)\cdot g)$. By strong approximation for $\PGL_2(\Q)$ this is a local question --- for $p\gg 1$ we apply Lemma \ref{at large primes relevant binary quartics are locally soluble even when twisted}, while for the remaining $p\ll 1$ of course $\frac{F(X,Y)}{B_*}\in \frac{1}{p}\cdot \Z[X,Y]$ already, so we are done with the claim.

Of course at the level of $\PGL_2(\Q)$-orbits this map amounts to the map $\PGL_2(\Q)\cdot F(X,Y)\mapsto \PGL_2(\Q)\cdot \frac{F(X,Y)}{B_*}$ so it is evidently a bijection, so that we are done.
\end{proof}

\begin{lem}\label{codimension two input for the uniformity estimate for quintic mordell curves}
Let $p\gg 1$. Let $F\in \Sym^4(2)(\Z)$ with $I(F) = 0$ be such that $-27 J(F)\cdot Z^2 = F(X,Y)$ is not soluble over $\Q_p$. Then: either $p\mid F$ or else $F$ has splitting type $(1^4)$ modulo $p$.

Consequently the image modulo $p$ of the set of $F\in \Sym^4(2)(\Z)$ with $I(F) = 0$ such that either $\PGL_2(\Z_p)\cdot F\neq \PGL_2(\Q_p)\cdot F\cap V(\Z_p)$ or else $-27 J(F)\cdot Z^2 = F(X,Y)$ is not soluble over $\Q_p$ lies in a codimension $2$ subvariety of $\{F : I(F) = 0\}$.
\end{lem}

\begin{proof}
Apply Lemma \ref{relevant local binary quartics have a root} for the first claim, and then repeat the argument given in Section \ref{the uniformity estimate for mordell curves section} for the second.
\end{proof}

Thus we may prove Theorem \ref{all six families} for $k\equiv 4\pmods{6}$ in the same way as for $k\equiv 5\pmods{6}$.

\begin{proof}[Proof of Theorem \ref{all six families} for $k\equiv 4\pmods{6}$.]
Repeat the proof of Theorem \ref{all six families} --- including the uniformity estimate (via Lemma \ref{codimension two input for the uniformity estimate for quintic mordell curves}) --- in the case $k\equiv 1\pmods{6}$ mutatis mutandis, with the exception that the local conditions (besides those modulo $\kappa^{O(1)}$) imposed be the everywhere local solubility of $-27 J(F)\cdot Z^2 = F(X,Y)$ --- rather than that of $Z^2 = F(X,Y)$ --- throughout. (Note that the fact that these provide the same local densities ultimately reduces to Lemma $3.20$ of Bhargava-Shankar's \cite{chaptertwo-bhargava-shankar-2-selmer}.)
\end{proof}

Note that we could have also used Corollary $6.20$ of Bhargava-Ho's \cite{chapterthree-bhargava-ho} and the discussion thereafter --- replacing $-27 J(F)\cdot Z^2 = F(X,Y)$ by $Z^2 = H(F)(X,Y)$ with $H(F)(X_1,X_2) := \disc_{(x,y)}\left(\det\left(x\cdot \frac{\d(F(X_1, X_2))}{\d X_1 \d X_j \d X_k} - y\cdot \frac{\d(F(X_1, X_2))}{\d X_2 \d X_j \d X_k}\right)_{j,k=1}^2\right)$ throughout. Indeed the argument is again the same save for the production of integral representatives, i.e.\ save for producing an $F\in \Sym^4(2)(\Z)^\nontriv$ for which $Z^2 = H(F)(X,Y)$ (or alternatively $\PGL_2(\Q)\cdot H(F)$) represents a given element of $\Sel_2(E_{0,B^4}/\Q) - \{0\}$, and we produce such integral representatives as follows.

\begin{lem}\label{integral representatives for the other way of treating quartic mordell curves}
There is an absolute constant $\kappa\in \Z - \{0\}$ such that the following holds. Let $0\neq B\in \Z$ be cubefree. Then: there is a bijection between $\Sel_2(E_{0,B^4}/\Q) - \{0\}$ and $$\PGL_2(\Q)\backslash \left\{F\in \frac{1}{\kappa}\cdot \Sym_4(2)(\Z)^\nontriv : \substack{I(F) = 0, J(F) = -2^6\cdot 3^3\cdot B, \\Z^2 = H(F)(X,Y)\text{ everywhere locally soluble}}\right\}.$$
\end{lem}

\begin{proof}
Factor $B =: B_*\cdot B_\square^2$ with $B_*$ squarefree --- thus $E_{0,B^4}\simeq E_{0,B_\square^2\cdot B_*^4}$ over $\Q$. As usual (i.e.\ just as in the proofs of Lemmas \ref{integral representatives for quintic mordell curves} and \ref{integral representatives for quartic mordell curves}) strong approximation for $\PGL_2/\Q$ reduces the problem to a local question. Now from the discussion after $(14)$ on page $14$ of Bhargava-Ho's \cite{chapterthree-bhargava-ho-website} we see that given an $(x,y)\in E_{0,B_\square^2\cdot B_*^4}(\Q_p)$ with $x,y\in \Z_p^\times$ (which by Lemma \ref{bad points don't actually matter} is the only case we must consider) there is a $\vec{v}\in (2\otimes \Sym_3(2))(\Z)^\nontriv$ such that $v_1 = ((\in \Z_p^\times), 0, 0, (\in \Z_p^\times))$ and $v_2 = (0, (\in \Z_p^\times), (\in \Z_p^\times), 0)$, and such that $F_{\vec{v}}(X,Y)$ is $\PGL_2(\Q_p)$-equivalent to the usual Kummer image $F_{(x,y)}(X,Y) := X^4 - 6x\cdot X^2 Y^2 + 8y\cdot X Y^3 - 3x^2\cdot Y^4$ (thus $a_6(\vec{v}) = 2\cdot B_\square\cdot B_*^2$). Then by the discussion after $(71)$ on page $65$ of Bhargava-Ho's \cite{chapterthree-bhargava-ho} we find an explicit $g\in M_2(\Z)$ with $\det{g} = a_6(\vec{v})$ such that $g\cdot \vec{v}$ is the image (under the evident map, i.e.\ under $\Sym_4(2)\inj 2\otimes \Sym_3(2)$) of a $G\in \Sym_4(2)(\Z)$, namely $-G(X,Y) = x^2\cdot X^4 + 4\widetilde{y}\cdot X^3 Y + 6x\cdot X^2 Y^2 + \frac{4x^2}{\widetilde{y}}\cdot X Y^3 + Y^4$, where $\widetilde{y} := y - \frac{a_6(\vec{v})}{2}$ (thus $\widetilde{y}^2 + a_6(\vec{v})\cdot \widetilde{y} = x^3$), which therefore has $I(G) = 0$ and $J(G) = 2^4\cdot 3^3\cdot a_6(\vec{v})^2 = 2^6\cdot 3^3\cdot B_\square^2\cdot B_*^4$. We are thus done for $p\ll 1$ or else for $p\nmid B_*$ by considering $-\frac{G(X,Y)}{B_*}$. Otherwise $p\gg 1$ and $p\mid B_*$, and then writing $\widetilde{G}(X,Y) := -G\left(Y, X - \frac{x^2}{\widetilde{y}}\cdot Y\right) =: X^4 + c\cdot X^2 Y^2 + \cdots + e\cdot Y^4$ by inspection --- namely $\widetilde{y}^2\equiv x^3\pmods{a_6(\vec{v})}$, so $\left(\frac{x^2}{\widetilde{y}}\cdot X + Y\right)^4\equiv -G(X,Y)\pmods{a_6(\vec{v})}$, and moreover $I(G) = 0$ --- we have that $a_6(\vec{v})\mid c, d$ and $a_6(\vec{v})^2\mid e$, so that $\frac{\widetilde{G}(p\cdot X,Y)}{p^3}\in \Z_p[X,Y]$. Thus $\widetilde{\widetilde{G}}(X,Y) := \frac{(\diag(p,1)\cdot \widetilde{G})(X,Y)}{B_*}\in \Z_p[X,Y]$ is such that $I(\widetilde{\widetilde{G}}) = 0$, $J(\widetilde{\widetilde{G}}) = -2^6\cdot 3^3\cdot B$, and such that $H(\widetilde{\widetilde{G}})(X,Y)$ is $\PGL_2(\Q_p)$-equivalent to $F_{(x,y)}(X,Y)$, so that $\widetilde{\widetilde{G}}$ is the desired integral representative.

We conclude as we did in the proof of Lemma \ref{integral representatives for quartic mordell curves}.
\end{proof}

\newpage

\renewcommand{\refname}{References.}

\bibliographystyle{amsplain}

\bibliography{quadricsinarithmeticstatistics}

\end{document}